\documentclass[11pt]{article}
    \usepackage[utf8]{inputenc}
    \usepackage[dutch,english]{babel}
    \usepackage{lipsum}
    \usepackage{mathtools}
    \usepackage{appendix}
    \usepackage{dsfont}
    \usepackage{graphicx}
	\usepackage{enumitem}
	\usepackage{mathrsfs}
    \usepackage{xfrac}
    \usepackage{amsmath, amsthm, amssymb}
    \usepackage{hyperref}
    \usepackage{import}
    \usepackage{upgreek}
    \usepackage{bigints}
    \usepackage{geometry}
    \theoremstyle{definition}
    \newtheorem{definition}{Definition}[section]
    \newtheorem{remark}{Remark}[section]
    
    \theoremstyle{plain}
    \newtheorem{theorem}{Theorem}[section]
    \newtheorem{proposition}[theorem]{Proposition}
    \newtheorem{lemma}[theorem]{Lemma}
    \newtheorem{corollary}[theorem]{Corollary}

    \newcommand{\beq}{\begin{eqnarray}}
    \newcommand{\eeq}{\end{eqnarray}}
    \newcommand{\bal}{\begin{align*}}
    \newcommand{\eal}{\end{align*}}
    \newcommand{\be}{\begin{equation}}
    \newcommand{\ee}{\end{equation}}
    \newcommand{\ben}{\begin{equation*}}
    \newcommand{\een}{\end{equation*}}
    \newcommand{\bl}{\begin{lemma}}
    \newcommand{\el}{\end{lemma}}
    \newcommand{\br}{\begin{remark}}
    \newcommand{\er}{\end{remark}}
    \newcommand{\bt}{\begin{theorem}}
    \newcommand{\et}{\end{theorem}}
    \newcommand{\bd}{\begin{definition}}
    \newcommand{\ed}{\end{definition}}
    \newcommand{\bind}{\begin{induction}}
    \newcommand{\eind}{\end{induction}}
    \newcommand{\bp}{\begin{proposition}}
    \newcommand{\ep}{\end{proposition}}
    \newcommand{\bc}{\begin{corollary}}
    \newcommand{\ec}{\end{corollary}}
    \newcommand{\bpr}{\begin{proof}}
    \newcommand{\epr}{\end{proof}}
    \newcommand{\bi}{\begin{itemize}}
    \newcommand{\ei}{\end{itemize}}
    \newcommand{\bsp}{\begin{split}}
    \newcommand{\esp}{\end{split}}

    \newcommand{\dd}{\mathop{}\!\mathrm{d}}
    \newcommand{\nn}{\nonumber}

    \newcommand{\caI}{{\mathcal I}}

    \newcommand{\scrM}{{\mathscr M}}
    \newcommand{\scrN}{{\mathscr N}}

    \newcommand{\Var}{\mathrm{Var}}
    \newcommand{\Cov}{\mathrm{Cov}}
    \newcommand{\E}{\mathbb{E}}
    \newcommand{\p}{\mathbb{P}}
    \newcommand{\Om}{\Omega}

\title{Stationary fluctuations of run-and-tumble particles}
\author{Frank Redig\footnote{Delft Institute of Applied Mathematics, TU Delft, Delft, The Netherlands. f.h.j.redig@tudelft.nl}, Hidde van Wiechen\footnote{Delft Institute of Applied Mathematics, TU Delft, Delft, The Netherlands. h.vanwiechen@tudelft.nl}}
\date{\today}

\begin{document}

\maketitle
\begin{center}
    \section*{Abstract}
    \end{center}

\begin{changemargin}{10mm}{10mm}
We study the stationary fluctuations of independent run-and-tumble particles.
We prove that the joint densities of particles with given internal state converges to an infinite dimensional Ornstein-Uhlenbeck process.
We also consider an interacting case, where the particles are subjected to exclusion.
We then study the fluctuations of the total density, which is a non-Markovian Gaussian process, and obtain its covariance in closed form.
By considering small noise limits of this non-Markovian Gaussian process, we obtain in a concrete example a large deviation rate function containing memory terms.
\end{changemargin}
\section{Introduction}
In this paper we consider a system of independent run-and-tumble particles on ${\mathbb Z}$ and study the stationary fluctuations of its empirical distribution.
Because particles have positions and internal states (which determine the direction in which they move and/or their rate of hopping over lattice edges), the hydrodynamic limit is a system of linear reaction-diffusion equations, describing the macroscopic joint evolution of the densities of particles with a given internal state.
In this sense, the paper can be viewed as a study of macroscopic properties of the multi-layer particle systems which we studied in \cite{Wiechen}.
The study of hydrodynamic limits and fluctuations around the hydrodynamic limit for particles with internal states, or alternatively, multi-layer systems is quite recent, and to our knowledge at present only a limited set of results is known: see \cite{erignoux}, \cite{floreani}, \cite{saada}, \cite{sasada2010hydrodynamic}.

Our interest in multi-layer systems is motivated  from the study  of active particles (see e.g. \cite{maes}), the study of double diffusivity models (see e.g. \cite{floreani} and references therein), and finally the study of particle systems described macrosopically by equations containing memory terms.
In this paper we consider multi-layer systems in which duality can be applied. Duality is a powerful tool which reduces the study of the hydrodynamic limit to the scaling limit of a single (dual) particle, and as we show in this paper (see Section \ref{stationary covariance} below) also determines uniquely the covariance of the stationary fluctuations of the empirical density of particles. Provided one can show that the stationary fluctuations converge to a  Gaussian limiting (distribution-valued) process, this limiting covariance uniquely determines the limiting stationary Gaussian process.

In our paper we prove that the fluctuation fields of the densities of particles with given internal state converge to a system of stochastic partial differential equations. In these limiting equations,  the drift is determined by the hydrodynamic limit, whereas the noise has both a conservative part coming from the transport of particles with a given internal state as well as a non-conservative part coming from the flipping of internal states.
We first deal with a system of independent particles, which has a simple dual consisting of independent particles with reversed velocities.
Next we indicate how to deal with interacting particles such as layered exclusion processes, where still duality can be used.

One of our motivations of studying fluctuation fields of particles with internal states is to understand fluctuation properties of the total density, i.e., disregarding the internal states of the particles.
The configuration which gives at each site the total number of particles is one of the simplest examples of a non-Markovian interacting particle system. The study of the hydrodynamic limit, fluctuations and large deviations around the hydrodynamic limit for non-Markovian particle systems is largely terra incognita.
Therefore, we believe that simple examples in which one can have some grip on the explicit form of fluctuations and large deviations are important to obtain.

In our setting, we prove that the fluctuations of the total density of particles converges to a Gaussian distribution-valued process which satisfies a non-Markovian SPDE.
We provide a concrete example where we can explicitly characterize the large deviations of the limiting SPDE in the small noise limit.
These large deviations give an indication of the  large deviations of the total density of particles. The latter can of course also be obtained via a contraction principle from the large deviations of the joint densities of particles
with a given internal state. However,  the large deviation rate function obtained via this contraction principle  is very implicit, and therefore in this paper we preferred not to follow this road in order to obtain an explicit form of the memory terms of the rate function.

The rest of our paper is organized as follows. In Section \ref{section 2} we introduce the run-and-tumble particle model and state preliminary results on ergodic measures, duality and hydrodynamic limit, the latter of which will be proven in the appendix \ref{appendix}. In Section \ref{section 3} we state the main result on stationary fluctuations for independent particles, Theorem \ref{fluctuations}, provide a direct proof of the limiting covariance in Section \ref{stationary covariance},  and consider an interacting case, namely a multi-layer version of the symmetric exclusion process, in Section \ref{4.3}. In Section \ref{section 5} we study the hydrodynamic limit and the fluctuations of the total density of particles, and prove a large deviations result for the limiting fluctuation process in a particular case.
In Section \ref{section 4} we prove the Theorem \ref{fluctuations}.

\section{Basic notations and definitions}\label{section 2}

In this paper we will look at the run-and-tumble particle process, which is a process designed to model active particles.  Let $V := \mathbb{Z}\times S$, with $S\subset \mathbb{Z}$ a finite set. The set $V$ is the state space of a single run-and-tumble particle. We see elements $v=(x,\sigma) \in V$ as particles with position $x \in \mathbb{Z}$ and internal state $\sigma \in S$. The dynamics of a single run-and-tumble particle are now as follows
\begin{itemize}
    \item[i.] At rate $\kappa N^2$ the particle performs a nearest neighbor jump, i.e., $(x,\sigma) \to (x\pm 1,\sigma)$
    \item[ii.] At rate $\lambda N$ the particle performs an active jump in the direction of its internal state, i.e., $(x,\sigma) \to (x+\sigma,\sigma)$.
    \item[iii.] At rate $c(\sigma,\sigma')$ the particle changes its internal state from $\sigma$ to $\sigma'$, i.e. $(x,\sigma) \to (x,\sigma')$. Here we assume that the rates $\left\{c(\sigma,\sigma') : \sigma,\sigma' \in S \right\}$ are irreducible and symmetric, i.e., $c(\sigma,\sigma') = c(\sigma',\sigma)$.
\end{itemize}
The run-and-tumble particle process is the process of configurations consisting of independent run-and-tumble particles. More precisely it is a Markov process $\{\eta_t:t\geq 0\}$ on the state space $\Om := \mathbb{N}^V$ consisting of independent random walkers on $V$ where every particle has the dynamics as described above.

From the dynamics we can write down the  following generator $L_N$ acting on local functions, i.e., functions $f: \Omega \to \mathbb{R}$ which only depend on a finite number of sites in $V$.
\begin{equation}\label{GEN}\begin{split}
    L_Nf(\eta) &=\kappa N^2 \sum_{(x,\sigma)\in V} \eta(x,\sigma) \left(f\big(\eta^{(x,\sigma)\to(x+1,\sigma)}\big)+f\big(\eta^{(x,\sigma)\to(x-1,\sigma)}\big)-2f(\eta)\right)\\
    &\ \ \ \ + \lambda N\sum_{(x,\sigma)\in V} \eta(x,\sigma) \left(f\big(\eta^{(x,\sigma) \to (x+\sigma,\sigma)}\big)-f(\eta)\right)\\
    &\ \ \ \ +   \sum_{(x,\sigma) \in V} \sum_{\sigma'\in S} \eta(x,\sigma)c(\sigma,\sigma') \left(f\big(\eta^{(x,\sigma)\to(x,\sigma')}\big) - f(\eta)\right).
\end{split}\end{equation}
Here $\eta(x,\sigma)$ denotes the number of particles at site $(x,\sigma)\in V$ in the configuration $\eta$, and  $\eta^{(x,\sigma)\to (y,\sigma')}$ denotes the configuration $\eta$ where a single particle has moved from $(x,\sigma)$ to $(y,\sigma')$.

With this choice of scaling, in the macroscopic limit, the densities of particles with a given internal state satisfy a system of linear reaction-diffusion equations (see section 5.1 below for the explicit form). Equivalently, one can view the choice of scaling as a diffusive time scale ($t\to N^2 t$), a weak asymmetry (active jumps in the direction of the velocity occur at rate $N= N^{-1} N^2$), and a slow reaction term (changes of internal state happen at rate $1=N^{-2} N^2$.
The scaling is also such that the motion of a single particle converges to a multi-layer Brownian motion with layer-dependent drift (cf. section 2.1 below).

\subsection{Scaling limit of the single particle dynamics}

We will  denote by $\mathscr{L}_N$ the Markov generator of a single run-and-tumble particle (rescaled in space),
more precisely, the generator of the process $(\tfrac{X_t}{N}, \sigma_t)$ where $X_t$ denotes the position and $\sigma_t$ the internal state of the particle.

This generator acts on a core consisting of test functions on the space $\mathbb{R}\times S$, which we denote by $C_{c,S}^\infty$, and which is  defined via
\[
C_{c,S}^\infty := \left\{\phi:\mathbb{R}\times S \to \mathbb{R} : \phi(\cdot,\sigma) \in C_c^\infty(\mathbb{R}) \text{ for all $\sigma \in S$} \right\}.
\]
The generator $\mathscr{L}_N$ then reads as follows:
\begin{align*}
\mathscr{L}_N \phi(x,\sigma) &= \kappa N^2 (\phi(x+\tfrac{1}{N},\sigma) + \phi(x-\tfrac{1}{N},\sigma) -2\phi(x,\sigma)) + \lambda N (\phi(x+\tfrac{\sigma}{N},\sigma) - \phi(x,\sigma))\\
&\ \ \ \ + \sum_{\sigma' \in S} c(\sigma,\sigma') (\phi(x,\sigma') - \phi(x,\sigma)).
\end{align*}
Corresponding to this generator we have  the corresponding Markov semigroup which we denote by $S^N_t$. Via Taylor approximation we obtain that $\mathscr{L}_N \phi \to A\phi$ uniformly as $N\to\infty$, where $A$ is the differential operator given by
\begin{equation}\label{A}
A\phi(x,\sigma) =   \left(\tfrac{\kappa}{2}\partial_{xx} +\sigma \lambda \partial_x\right)\phi(x,\sigma) + \sum_{\sigma' \in S}c(\sigma,\sigma') \big(\phi(x,\sigma') - \phi(x,\sigma)\big).
\end{equation}
Because $A$ generates a Markov semigroup as well, as a consequence of the convergence of the generators we can also obtain $S^N_t\phi \to e^{tA}\phi$ uniformly for all $\phi \in C_{0,S}$, i.e., the functions space consisting of functions $\phi:\mathbb{R}\times S\to \mathbb{R}$ such that $\phi(\cdot,\sigma) \in C_0(\mathbb{R})$ for all $\sigma \in S$.

The operator $A$ above is also an operator on (a subset of) the Hilbert space $L^2(\dd x\times |\cdot|_S)$, where $|\cdot|_S$ is the counting measure over $S$. The  inner product on this Hilbert space, denoted by $\left<\left<\cdot,\cdot\right>\right>$, is the following
\begin{equation}\label{innerprod}
\left<\left<\phi,\psi\right>\right> := \sum_{\sigma \in S} \int_\mathbb{R} \phi(x,\sigma)\psi(x,\sigma)\dd x.
\end{equation}
Later on we will need   the adjoint of the operator $A$ with respect to this inner product, which acts on $\phi\in C_{c,S}^\infty$ as follows:
\begin{equation}\label{A*}
A^*\phi(x,\sigma) =   \left(\tfrac{\kappa}{2}\partial_{xx} -\sigma \lambda \partial_x\right)\phi(x,\sigma) + \sum_{\sigma' \in S}c(\sigma,\sigma') \big(\phi(x,\sigma') - \phi(x,\sigma)\big).
\end{equation}

\subsection{Basic properties of independent run-and-tumble particles}
Before we state the theorem of the stationary fluctuations, we first review a few  known results on run-and-tumble particles which we need.
\subsubsection{Stationary ergodic product measures}
We define the measures $\mu_\rho$, with $\rho \in [0,\infty)$, as the product Poisson measure with density $\rho$, i.e.
\begin{equation*}
    \mu_\rho := \bigotimes_{(x,\sigma)\in V} \mathrm{Pois}(\rho).
\end{equation*}
In \cite{Wiechen} it is proved  that these measures   are stationary and ergodic with respect for run-and-tumble particle process $\{\eta_t: t\geq0 \}$.  For this reason, when we study the stationary fluctuations of the densities of particles with given internal state, we will start the process $\{\eta_t: t\geq 0\}$ from the measure $\mu_\rho$.

\subsubsection{Duality}\label{duality sec}
\begin{definition}
We say that two Markov processes $\{\eta_t : t\geq0\}$ and $\{\xi_t : t\geq0\}$,  on the state spaces $\Om$ and $\Om'$ respectively, are $\emph{dual}$ to one another with respect to a duality function  $D:\Om\times \Om' \to \mathbb{R}$  if
\begin{equation}\label{duality def}
    \E_\eta \left[D(\xi,\eta_t)\right] = \widehat{\E}_\xi \left[D(\xi_t,\eta)\right] < \infty,
\end{equation}
where $\E_\eta$ denotes the expectation in $\{\eta_t :  t \geq 0\}$ starting from $\eta$ and  $\widehat{\E}_\xi$ the expectation in the dual process $\{\xi_t :  t\geq 0\}$ starting from $\xi$.
\end{definition}
In \cite{Wiechen} it is proved that the run-and-tumble particle process is dual to its time-reversed process where the active jumps are in the reverse direction, i.e., the process corresponding to the following generator
\begin{equation*}\begin{split}
    \widehat{L}_Nf(\eta) &=\kappa N^2 \sum_{(x,\sigma)\in V} \eta(x,\sigma) \left(f\big(\eta^{(x,\sigma)\to(x+1,\sigma)}\big)+f\big(\eta^{(x,\sigma)\to(x-1,\sigma)}\big)-2f(\eta)\right)\\
    &\ \ \ \ + \lambda N\sum_{(x,\sigma)\in V} \eta(x,\sigma) \left(f\big(\eta^{(x,\sigma) \to (x-\sigma,\sigma)}\big)-f(\eta)\right)\\
    &\ \ \ \ +   \sum_{(x,\sigma) \in V} \sum_{\sigma'\in S} \eta(x,\sigma)c(\sigma,\sigma') \left(f\big(\eta^{(x,\sigma)\to(x,\sigma')}\big) - f(\eta)\right).
\end{split}\end{equation*}
The duality function is then given by
\begin{equation*}
    D(\xi,\eta) = \prod_{(x,\sigma) \in V} \frac{\eta(x,\sigma)!}{\xi(x,\sigma)!(\eta(x,\sigma) - \xi(x,\sigma))!}\cdot I\big(\xi(x,\sigma) \leq \eta(x,\sigma)\big),
\end{equation*}
where $I$ denotes the indicator function, and where $\xi$ is assumed to be a finite configuration, i.e.,
\[
\sum_{(x,\sigma)} \xi(x,\sigma) <\infty
\]

In our paper we will mostly need this duality relation in the form of duality with a single dual particle, i.e.,
\[
\E_\eta[\eta_t(x,\sigma)] = \widehat{\E}_{(x,\sigma)}[\eta(\widehat{X}_t,\widehat{\sigma}_t)],
\]
where  $(\tfrac{\widehat{X}_t}{N},\widehat{\sigma}_t)$ is the process corresponding to the (time-reversed) generator $\widehat{\mathscr{L}}_N$ given by
\begin{align*}
\widehat{\mathscr{L}}_N \phi(x,\sigma) &= \kappa N^2 (\phi(x+\tfrac{1}{N},\sigma) + \phi(x-\tfrac{1}{N},\sigma) -2\phi(x,\sigma)) + \lambda N (\phi(x-\tfrac{\sigma}{N},\sigma) - \phi(x,\sigma))\\
&\ \ \ \ + \sum_{\sigma' \in S} c(\sigma,\sigma') (\phi(x,\sigma') - \phi(x,\sigma)).
\end{align*}
We denote the corresponding Markov semigroup of this process as $\widehat{S}_t^N$. By a Taylor expansion, we obtain that $\widehat{\mathscr{L}}_N\phi \to A^*\phi$, with $A^*$ defined as in \eqref{A*}, uniformly in $N$ for all $\phi \in C_{c,S}^\infty$,  and therefore we are able to write  for all $\phi \in C_{0,S}$ that  $\widehat{S}^N_t\phi \to e^{tA^*}\phi$ uniformly.

\subsection{Hydrodynamic limit}
In this section we will briefly mention the hydrodynamic limit of the run-and-tumble particle process.
For the proof, which follows standard methodology, we refer to the appendix.

Given a function $\rho:\mathbb{R}\times S\to \mathbb{R}$ such that $\rho(\cdot,\sigma) \in C_b^2(\mathbb{R})$ for all $\sigma \in S$, we start by defining the product Poisson measures $\mu_\rho^N$ for every $N\in\mathbb{N}$ as follows
\begin{equation}\label{loceq}
\mu_\rho^N := \bigotimes_{(x,\sigma) \in V} \mathrm{Pois}\big(\rho(\tfrac{x}{N},\sigma)\big).
\end{equation}
This is the local equilibrium distribution corresponding to the macroscopic profile $\rho$.

Furthermore, for every $N\in\mathbb{N}$, the process $\{\eta_t^N:t\geq0\}$ is the run-and-tumble particle process started from $\eta_0^N \sim \mu_\rho^N$.
We can now define the empirical measures of the process, denoted by  $\pi^N = \left\{\pi_{t}^N : t\geq 0\right\}$,  as follows
\begin{equation}\label{empi}
    \pi^N_{t} := \frac{1}{N}\sum_{(x,\sigma)\in V} \eta^N_t(x,\sigma)\delta_{(\tfrac{x}{N},\sigma)},
\end{equation}
where $\delta$ is the dirac measure.
We think of
$\pi^N_{t}$ as the macroscopic profile corresponding to the microscopic configuration $\eta_t$.
In the rhs of \eqref{empi} every particle of type $\sigma$ contributes a mass $1/N$ at the ``macro spatial location'' $x/N$.

For every $t\geq0$, $\pi^N_{t}$ is a positive measure on $\mathbb{R}\times S$ such that when paired with a test function $\phi \in C_{c,S}^\infty$ we obtain
\begin{equation*}
    \pi^N_{t}(\phi) := \left<\phi, \pi^N_{t}\right> = \frac{1}{N} \sum_{(x,\sigma)\in V} \eta^N_t(x,\sigma) \phi(\tfrac{x}{N},\sigma).
\end{equation*}
By the choice of the initial distribution, we have at time $t=0$ zero that
\[
\pi^N_0(\phi)\to \int \rho(x,\sigma) \phi(x,\sigma) dx
\]
We then have the following result for the hydrodynamic limit.
\begin{theorem}\label{hydrodynamic limit}
For every $t\geq0$, $\varepsilon>0$ and $\phi \in C_{c,S}^\infty$, we have that
\begin{equation*}
    \lim_{N\to\infty} \p\left(\left|\pi_{t}^N(\phi) - \sum_{\sigma \in S}\int \rho_t(x,\sigma ) \phi(x,\sigma)dx\right| >\varepsilon\right)=0,
\end{equation*}
where $\rho_t(x,\sigma)$ solves the PDE  $\dot{\rho}_t = A^*\rho_t$ with initial condition $\rho_0(x,\sigma) = \rho(x,\sigma)$.
\end{theorem}

This results is actually a corollary of an stronger theorem which shows convergence of the trajectories $\pi^N$ in the path space $D([0,T];\textbf{M})$ equipped with the Skorokhod topology, where $\textbf{M}$ is the space of Radon measures on $\mathbb{R}\times S$.
Let $\pi = \{\pi_t:t\geq 0\}$ denote the trajectory of measures on $\mathbb{R}\times S$ such that for all $t\geq0, \phi\in C_{c,S}^\infty$ we have that $\left<\phi, \pi_t\right> = \left<\left<\phi, \rho_t\right>\right>$, where $\rho_t$ solves the PDE in the above theorem. The trajectory $\pi$  is then the unique continuous path in $D([0,T];\textbf{M})$ such that for all $\phi \in C_{c,S}^\infty$
\begin{equation}\label{pi}
    \scrM^{\phi}_t(\pi)= \pi_t(\phi) - \pi_0(\phi) - \int_0^t \pi_s(A\phi)\dd s=0.
\end{equation}

\begin{theorem}\label{weak solution}
For any $N\in\mathbb{N}$, let $P^N$ be the law of the process $\pi^N$. Then $P^N \to \delta_{\pi}$ weakly in $D([0,T]; \emph{\textbf{M}})$ for any $T>0$, with $\pi$ the unique continuous path solving \eqref{pi}.
\end{theorem}
For the sake of self-containedness, the  proof of Theorem \ref{weak solution} is provided in the appendix. The method of proof is standard and it follows Sepp\"al\"ainen, in \cite[Chapter 8]{timo}.

\subsection{Fluctuation fields}
For every $N\in\mathbb{N}$, we define the fluctuation field $Y^N := \{Y_{t}^N: t\geq 0\}$  as
\begin{equation}\label{fluctuation field}
    Y^N_{t} = \frac{1}{\sqrt{N}} \sum_{x \in \mathbb{Z}} \big(\eta_t(x,\sigma)- \rho\big)\delta_{(\tfrac{x}{N},\sigma)}.
\end{equation}
This process takes values in the space of distributions on $\mathbb{R}\times S$, denoted by $(C_{c,S}^\infty)^*$. We expect the fluctuation field $Y^N$ to converge weakly to a generalized stationary Ornstein-Uhlenbeck process. Before we can state the result we first recall some basic definitions of space-time white noise (see e.g. \cite{hida} for a detailed account).
\begin{definition}
A random distribution $\mathscr{W}$   is called a \emph{white noise} on $\mathbb{R}\times S$ if $\{\left<\phi, \mathscr{W}\right> : \phi \in C^\infty_{c,s}\}$ is jointly centered Gaussian with covariance
\[
\E[\left<\phi, \mathscr{W}\right> \left<\psi, \mathscr{W}\right>] = \left<\left<\phi,\psi\right>\right>.
\]
where $\langle\langle,\rangle\rangle$ denotes the inner product defined in \eqref{innerprod}.
We denote by $\dd \mathscr{W}_t$ the time-differential of space-time white noise. This object is such that when paired with a test function $\phi \in C_{c,S}^\infty$ and integrated over time gives a Brownian motion, i.e.,
\[
\int_0^t \left<\phi ,\dd \mathscr{W}_s\right> = B(\left<\left<\phi, \phi \right>\right>t),
\]
where $B(\cdot)$ is a standard Brownian motion on $\mathbb{R}$.
We denote by $\frac{\dd \mathscr{W}_t}{\dd t}$ the corresponding space-time white noise. This random space-time distribution is such that for all $\phi: [0,T]\times\mathbb{R}\times S\to\mathbb{R}$, with $\phi(t, \cdot)$ a test function
$\langle\phi, \frac{\dd \mathscr{W}_t}{\dd t}\rangle$ is jointly Gaussian with covariance
\[
\E\left[\left\langle\phi, \frac{\dd \mathscr{W}_t}{\dd t}\right\rangle \left\langle\psi, \frac{\dd \mathscr{W}_t}{\dd t}\right\rangle\right]=\int_0^T
\left\langle\left\langle\phi(t, \cdot), \psi(t, \cdot)\right\rangle\right\rangle \dd t.
\]

\end{definition}
\begin{remark}
Informally speaking, a white noise on $\mathbb{R}\times S$ is a Gaussian field
$W(x,\sigma)$ with covariance
$\delta(x-y)\delta_{\sigma,\sigma'}$, and a space-time white noise on $\mathbb{R}\times S$
is a Gaussian field
$W(t,x,\sigma)$ with covariance $\delta(t'-t)\delta(x-y)\delta_{\sigma,\sigma'}$.
\end{remark}

\section{Stationary fluctuations}\label{section 3}
We are now ready to state our result on stationary fluctuations. We start with the case of independent particles;  in Section \ref{4.3} below we will consider an interacting case.
\begin{theorem}\label{fluctuations}
Assume that $\eta_0$ is distributed according to the Poisson product measure $\mu_\rho$.
For every $N\in\mathbb{N}$, let $Q^N$ denote the law of the process $Y^N$ defined in
\eqref{fluctuation field}. Then $Q^N \to Q$ weakly in $D([0,T]; (C_{c,S}^\infty)^*)$ for any $T>0$, where $Q$ is the law of the stationary Gaussian process $Y$ satisfying the following SPDE
\begin{equation}\label{SPDE}
    \dd Y_t = A^*Y_t\dd t + \sqrt{2 \kappa\rho}\partial_{x}\dd \mathscr{W}_t + \sqrt{2 \rho\Sigma}\dd \mathscr{\tilde{W}}_t.
\end{equation}
Here $\dd \mathscr{W}_t$ and $\dd \mathscr{\tilde{W}}_t$ are two independent space-time white noises on the space $\mathbb{R}\times S$, and  $\Sigma$ is the operator working on test functions $\phi\in C_{c,S}^\infty$ as
\begin{equation}\label{Sigma}
(\Sigma\phi)(x,\sigma) = -\sum_{\sigma'\in S} c(\sigma,\sigma') \big(\phi(x,\sigma') - \phi(x,\sigma)\big).
\end{equation}
\end{theorem}
By the assumed symmetry of the rates $c(\sigma,\sigma')$, for $\phi,\psi \in C_{c,S}^\infty$ we have $\left<\left<\Sigma\phi,\psi\right>\right> = \left<\left<\phi,\Sigma\psi\right>\right>$,
and moreover $\left<\left<\Sigma\phi,\psi\right>\right>\geq 0$. Hence the operator is bounded, self-adjoint and non-negative and therefore
its square root $\sqrt{\Sigma}$ is well-defined.
The process $\partial_x \dd \mathscr{W}_t$ is defined as the process of distributions such that for all $\phi \in C_{c,S}^\infty$
\[
\left<\phi, \partial_x \dd \mathscr{W}_t\right> = - \left<\partial_x \phi, \dd \mathscr{W}_t\right>.
\]

The rigorous meaning of the SPDE in \eqref{SPDE} is defined in terms of a martingale problem
as in \cite{Kipnis}. More precisely, the map $\phi \mapsto Y_t(\phi)$ is the solution of the following martingale problem: for every $\phi \in C_{c,S}^\infty$, the following two processes
\begin{equation}\label{mart prob}
\begin{split}
\mathscr{M}^{\phi}_t(Y) &=Y_t(\phi) - Y_0(\phi) - \bigintsss_0^t Y_{s}(A\phi)ds,\\
\mathscr{N}^{\phi}_t(Y) &=\mathscr{M}^{\phi}_t(Y)^2 - 2t\kappa \rho  \left<\left<\partial_{x} \phi,\partial_{x} \phi\right>\right>- 2t\rho \left<\left<\phi, \Sigma\phi\right>\right>
\end{split}
\end{equation}
 are  martingales with respect to the natural filtration $\mathscr{F}_t = \sigma(Y_s:0\leq s\leq t)$.

\subsection{Stationary covariance of the fluctuation fields via duality}\label{stationary covariance}
We will first compare the covariance structure of the limiting process of $Y^N$ with the covariance structure of the process solving the SPDE in \eqref{SPDE}. This covariance uniquely characterizes the process. More precisely, if we can prove
that $Y^N_t\to Y_t$ where $Y_t$ is a distribution-valued stationary Gaussian process, then the covariance $\E(Y_t(\phi) Y_0(\psi))$ uniquely determines this process. In that sense, the computation of the covariance already determines the only possible candidate limit $Y_t$. As we show below, the covariance is in turn completely determined by the scaling limit of a single dual particle. This shows that for systems with duality, both the hydrodynamic limit and the stationary fluctuations are uniquely determined by the scaling limit of a single dual particle.
\begin{proposition}\label{covariance structure}
For all $\phi,\psi \in C_c^\infty(\mathbb{R}\times S)$
\[
\lim_{N\to\infty} \E[Y_t^N(\phi) Y_0^N(\psi)] = \E[Y_t(\phi) Y_0(\psi)] = \rho \cdot \left<\left<e^{tA}\phi, \psi\right>\right>.
\]
Here $\E$ denotes the stationary expectation starting from the initial configuration distributed according to $\eta_0\sim \mu_\rho$.
\end{proposition}
\begin{proof}
If $Y$ is a solution to the SPDE in \eqref{SPDE}, then we can write
\[
Y_t(\phi) = \mathscr{M}_t^\phi(Y) + Y_0(\phi) + \int_0^tY_s(A\phi)\dd s,
\]
where $\mathscr{M}_t^\phi(Y)$ is a martingale with respect to the filtration $\mathscr{F}_t = \sigma\left(Y_s: 0\leq s\leq t\right)$ such that $\mathscr{M}_0^\phi(Y)=0$. By the martingale property we have that
\[
\E[\mathscr{M}_t^\phi(Y) Y_0(\psi)] = \E\big[\E[\mathscr{M}_t^\phi(Y) Y_0(\psi)|\mathscr{F}_0]\big] = \E\big[ Y_0(\psi)\E[\mathscr{M}_t^\phi(Y)|\mathscr{F}_0]\big]=0,
\]
and so
\begin{align*}
    \E[Y_t(\phi) Y_0(\psi)] = \E[Y_0(\phi)Y_0(\psi)] + \int_0^t \E[Y_s(A\phi) Y_0(\psi)]\dd s.
\end{align*}
Therefore, using that $ \E[Y_0(\phi)Y_0(\psi)]=\left<\left<\phi, \psi\right>\right>$ we obtain that if $Y$ is a solution of \eqref{SPDE}, then we have
\[
\E[Y_t(\phi) Y_0(\psi)]= \E[Y_0(e^{tA}\phi) Y_0(\psi)] = \rho \cdot \left<\left<e^{tA}\phi, \psi\right>\right>.
\]
On the other hand, for any $N\in\mathbb{N}$ we have that
\begin{align}\label{eq1}
\nn\E\left[Y^N_t(\phi) Y^N_0(\psi)\right]
\nn&=\frac{1}{N} \sum_{(x,\sigma)\in V}\sum_{(y,\sigma')\in V} \phi(\tfrac{x}{N},\sigma) \psi(\tfrac{y}{N},\sigma')\int\E_\eta\left[ (\eta_t(x,\sigma)-\rho)(\eta(y,\sigma')-\rho)\right]\dd\mu_\rho(\eta)\\
\nn&=\frac{1}{N} \sum_{(x,\sigma)\in V}\sum_{(y,\sigma')\in V} \phi(\tfrac{x}{N},\sigma) \psi(\tfrac{y}{N},\sigma')\int\widehat{\E}_{(x,\sigma)}\left[ (\eta(\widehat{X}_t,\widehat{\sigma}_t)-\rho)(\eta(y,\sigma')-\rho)\right]\dd\mu_\rho(\eta)\\
&=\frac{1}{N} \sum_{(x,\sigma)\in V}\sum_{(y,\sigma')\in V} \phi(\tfrac{x}{N},\sigma) \psi(\tfrac{y}{N},\sigma')\widehat{\E}_{(x,\sigma)}\left[\mathrm{Cov}_{\mu_\rho}\left( \eta(\widehat{X}_t,\widehat{\sigma}_t), \eta(y,\sigma')\right)\right],
\end{align}
where we used duality for the second equality and Fubini for the last equality. Now note that, because $\mu_\rho$ is a product of Poisson measures, the covariance term is equal to $\rho$ if and only if $(\widehat{X}_t,\widehat{\sigma}_t) = (y,\sigma')$ and zero otherwise. Therefore
\begin{align}\label{eq2}
\nn \sum_{(y,\sigma')\in V} \psi(\tfrac{y}{N},\sigma')\widehat{\E}_{(x,\sigma)}&\left[\mathrm{Cov}_{\mu_\rho}\left( \eta(\widehat{X}_t,\widehat{\sigma}_t), \eta(y,\sigma')\right)\right]\\
\nn &= \rho \sum_{(y,\sigma')\in V} \psi(\tfrac{y}{N},\sigma')\widehat{\E}_{(x,\sigma)}\left[  I\left((\widehat{X}_t,\widehat{\sigma}_t) = (y,\sigma')\right) \right]\\
&= \rho \cdot (\widehat{S}_t^N \psi)(\tfrac{x}{N},\sigma).
\end{align}
Here $\widehat{S}_t^N$ is the semigroup of the Markov process $(\tfrac{\widehat{X}_t}{N},\widehat{\sigma}_t)$, for which we have the following uniform convergence $\widehat{S}_t^N \psi \to e^{tA^*}\psi$ (see Section \ref{duality sec}) . By now combining \eqref{eq1} and \eqref{eq2}, we  find that
\[
\E\left[Y^N_t(\phi) Y^N_0(\psi)\right]  = \rho \cdot \frac{1}{N} \sum_{(x,\sigma)\in V}\sum_{(y,\sigma')\in V} \phi(\tfrac{x}{N},\sigma)(\widehat{S}_t^N \psi)(\tfrac{x}{N},\sigma) \to \rho \cdot \left<\left< \phi, e^{tA^*}\psi\right>\right> = \rho \cdot \left<\left<e^{tA}\phi,\psi\right>\right>,
\]
which concludes the proof.
\end{proof}
\begin{remark}
In Proposition 3.2, the only place where the independence of the particles is manifest is in the pre-factor $\rho$ which corresponds to the limiting variance of the fluctuation field at time zero, because $\eta_0$ is distributed as $\mu_\rho$. When considering any other system which satisfies duality, when $A$ is the scaling limit of the single particle generator, and $\chi(\rho)$ is the limiting variance of the fluctuation field at time zero, we find that the limiting covariance is given by
\[
\E[Y_t(\phi) Y_0(\psi)]= \chi(\rho) \left<\left<e^{tA} \phi. \psi\right>\right>.
\]
E.g. for the exclusion process studied in the section below, $\chi(\rho)=\rho(\alpha-\rho)$.
\end{remark}
\subsection{Interacting case: the  multi-layer SEP}\label{4.3}
The multi-layer symmetric exclusion process, or multi-layer SEP, is a generalization of the symmetric exclusion process on $\mathbb{Z}$ to the multi-layered setting on $\mathbb{Z}\times S$. For this process we look at configurations $\eta \in \{0,1, ..., \alpha\}^V$ with $\alpha \in \mathbb{N}$, i.e., there are  at most $\alpha$ particles per site $v\in V$. Instead of having an active component on every layer $\sigma \in S$ like the run-and-tumble particle system, multi-layer SEP switches to a  different diffusion co\"efficient, denoted by $\kappa_\sigma$, between the layers. The generator of this process is then as follows
\begin{equation*}
    \begin{split}
    L^{SEP}_Nf(\eta) &= N^2 \sum_{(x,\sigma)\in V} \kappa_\sigma \sum_{|x-y|=1}  \eta(x,\sigma) \left(\alpha-\eta(y,\sigma)\right)\left(f\big(\eta^{(x,\sigma)\rightarrow(y,\sigma)}\big)-f(\eta)\right)\\
 &\ \ \ \ \ \ \ \ +   \sum_{(x,\sigma) \in V} \sum_{\sigma'\in S}c(\sigma,\sigma') \eta(x,\sigma) (\alpha-\eta(x,\sigma'))\left(f\big(\eta^{(x,\sigma)\rightarrow(x,\sigma')}\big) - f(\eta)\right).
\end{split}
\end{equation*}
In \cite{Wiechen} it is proved that this process is self-dual and has ergodic measures given by product Binomial measures $\nu_\rho = \bigotimes_{v\in V} \mathrm{Bin}(\alpha, \rho)$ where $\rho\in (0,1)$ is  constant.

The corresponding single-particle generator is then given by
\begin{equation*}
    \mathscr{L}_N^{SEP}\phi(x,\sigma) = \alpha \kappa_\sigma \big((\phi(x+\tfrac{1}{N},\sigma) + \phi(x-\tfrac{1}{N},\sigma) - 2\phi(x,\sigma)\big) + \sum_{\sigma' \in S} c(\sigma,\sigma') \big( \phi(x,\sigma') - \phi(x,\sigma)\big),
\end{equation*}
and $\mathscr{L}_N^{SEP}\phi \to B\phi$ uniformly, where
\begin{equation*}
    (B\phi)(x,\sigma) = \frac{\alpha \kappa_\sigma}{2}\partial_{xx} \phi(x,\sigma) + \sum_{\sigma' \in S} \alpha c(\sigma,\sigma') \big( \phi(x,\sigma') - \phi(x,\sigma)\big).
\end{equation*}
Since we took the rates $c(\sigma,\sigma')$  symmetric, this operator is self-adjoint in the Hilbert space $L^2(\dd x\times |\cdot|_S)$.

Using the same line of proof as in Section \ref{section 4} below, we obtain the following SPDE for the stationary fluctuation field,
\begin{equation}
    \dd Y_t = B Y_t\dd t + \sqrt{2 \rho(\alpha-\rho)K}\partial_{x}\dd \mathscr{W}_t + \sqrt{2 \rho(\alpha-\rho)\Sigma}\dd \mathscr{\tilde{W}}_t.
\end{equation}
Here $K$ is the operator given by $(K\phi)(x,\sigma) = \kappa_\sigma \phi(x,\sigma)$. Note in the noise terms the appearance of the terms $\rho(\alpha-\rho)$ instead of $\rho$ as in \eqref{SPDE}. This comes from the fact that for $(x,\sigma)\neq (y,\sigma')$
\[
\E_{\nu_\rho}[\eta_s(x,\sigma)(\alpha-\eta_s(y,\sigma'))] = \rho(\alpha-\rho),
\]
which plays a role in the calculation of the expectation of the Carr\'e du champ operator.

\section{Scaling limits of the total density}\label{section 5}
If we sum over the layers, i.e., over the $\sigma$-variables, then the resulting configuration which gives the total number of particles at each site is no longer a Markov process. Therefore, both in the hydrodynamic limit as well as in the fluctuations we expect  memory terms to appear in the form of higher order time derivatives in the limiting equations.
The stationary fluctuations of the empirical distribution of the total number of particles will then become a non-Markovian  Gaussian process which we can identify explicitly.

Next, we consider the small-noise limit of these fluctuations. We then obtain a large deviation principle via large deviations of Schilder's type for Gaussian processes
(i.e., small variance limit of Gaussian processes, see e.g. \cite{DS} p.\ 88, and also \cite{lifshits}), and we have  memory terms in the corresponding large deviation rate function. We will make these memory effects explicit in the simplest possible setting where $\kappa=0$ in \eqref{GEN}. To our knowledge, this is the first example of an explicit expression for a large deviation rate function of the empirical distribution of particles in a  non-Markovian context. In general such rate functions can be obtained from the contraction principle of the Markovian multi-layer system, but this expression in the form of an infimum is implicit, can rarely be  made explicit, and therefore does not make manifest the effect of memory terms.

In the whole of this section, for notational simplicity,   we further restrict to $S=\{-1,1\}$ (two layers) and put $c(1,-1) = c(-1,1)=:\gamma$. The aim is then to study the fluctuations of the total density of particles, where we sum up the particles in both layers. This produces an empirical measure and  fluctuation field on $\mathbb{R}$ given by
\begin{equation*}
    \zeta^N_t = \frac{1}{N} \sum_{(x,\sigma) \in V} \eta^N_t(x,\sigma) \delta_{\tfrac{x}{N}}, \ \ \ \ \ \ Z^N_t = \frac{1}{\sqrt{N}} \sum_{(x,\sigma)\in V} (\eta_t(x,\sigma) - \rho)\delta_{\tfrac{x}{N}}.
\end{equation*}
\subsection{Hydrodynamic equation for the total density}
From Theorem \ref{hydrodynamic limit} we can deduce that $\zeta_t^N$ converges in probability to $\varrho_t(x)\dd x$, where the density $\varrho_t(x)$ is the sum of the densities on both layers, i.e., $\varrho_t(x) = \rho_t(x,1) + \rho_t(x,-1)$ with $\rho_t(x,\sigma)$ the solution to the hydrodynamic equation $\dot{\rho}_t = A^*\rho_t$. We can rewrite this equation as a coupled system of linear PDE's given by
\[
\begin{cases}
    \dot{\rho}_t(x,1) = \left(\frac{\kappa}{2}\partial_{xx} - \lambda\partial_x\right) \rho_t(x,1) + \gamma(\rho_t(x,-1) - \rho_t(x,1)),\\\\
    \dot\rho_t(x,-1) = \left(\frac{\kappa}{2}\partial_{xx} + \lambda\partial_x\right) \rho_t(x,-1) + \gamma(\rho_t(x,1) - \rho_t(x,-1)).
\end{cases}
\]
Summing up both equations gives us a  PDE for the total density $\varrho_t(x)$. This PDE  also depends on the difference of the densities, which we will denote by $\Delta_t(x) := \rho_t(x,1) - \rho_t(x,-1)$, and therefore we get a new system of PDE's,
\begin{equation}\label{system sum}
\begin{cases}
\dot{\varrho}_t(x) = \frac{\kappa}{2} \partial_{xx} \varrho_t(x) - \lambda \partial_x \Delta_t(x),\\\\
\dot{\Delta}_t(x) = \frac{\kappa}{2} \partial_{xx} \Delta_t(x) - \lambda \partial_x\varrho_t(x) -2\gamma \Delta_t(x).
\end{cases}
\end{equation}
From this system we can actually find a closed equation for $\varrho(x)$. Namely, by first taking a second derivative in time of the upper equation we find that
\begin{align}\label{kaki}
\ddot{\varrho}_t(x)
= \frac{\kappa}{2} \partial_{xx} \dot{\varrho}_t(x) - \lambda \partial_x \dot{\Delta}_t(x)
= \frac{\kappa}{2} \partial_{xx} \dot{\varrho}_t(x) - \lambda \partial_x\left( \frac{\kappa}{2} \partial_{xx} \Delta_t(x) - \lambda \partial_x\varrho_t(x) -2\gamma \Delta_t(x)\right).
\end{align}
Now we use that from the first equation in \eqref{system sum} we have  $-\lambda \partial_x \Delta_t(x) = \dot{\varrho}_t(x) - \frac{\kappa}{2}\partial_{xx} \varrho_t(x)$. Substituting this in \eqref{kaki},  we find the following closed equation for the total density
\[
\ddot{\varrho}_t(x) - (\kappa \partial_{xx}+2\gamma)\dot{\varrho}_t(x) = \left((\lambda^2-\gamma\kappa) \partial_{xx} -  \frac{\kappa^2}{4}  (\partial_x)^4\right) \varrho_t(x).
\]
\subsection{Fluctuations of the total density}
For the analysis of the fluctuation field of the total density we first define the fluctuation fields of each layer, and then by taking higher order derivatives as in the previous subsection, we obtain a second order SPDE for the fluctuations of the total density (cf. \eqref{Zt} below). We first   set up a framework where we can rigorously deal with the various distributions  coming from the SPDE given in \eqref{SPDE} corresponding to both layers. We start by defining a fluctuation field for each layer individually.
\[
Y_{t,\sigma}^N = \frac{1}{\sqrt{N}} \sum_{x\in \mathbb{Z}} (\eta_t(x,\sigma) - \rho) \delta_{\tfrac{x}{N}}, \ \sigma\in \{-1,1\}
\]
The relation between these  fluctuation fields and $Z^N_t$ is as follows: for every  $\phi \in C_c^\infty$ we have that
\begin{equation}\label{ping}
\left<\phi, Z^N_t\right> = \left<\phi, Y_{t,1}^N\right> + \left<\phi, Y_{t,-1}^N\right>.
\end{equation}
However, there is also a direct relation between the fluctuation fields on both layers and the fluctuation field $Y^N_t$ on $\mathbb{R}\times S$ defined in \eqref{fluctuation field}: for every $\phi\in C_{c,S}^\infty$ the following holds
\begin{equation}\label{pong}
\left<\phi,Y_t^N\right> = \left<\phi(\cdot,1), Y_{t,1}^N\right> + \left<\phi(\cdot,-1), Y_{t,-1}^N\right>.
\end{equation}
In this way $Y_t^N$, but more importantly  its limiting process $Y_t$, can be interpreted as a column vector of distributions, $Y_t = \left(\begin{matrix} Y_{t,1}&Y_{t,-1}\end{matrix}\right)^T$, working on a row vector of functions, $\phi = \left(\begin{matrix} \phi(\cdot,1)&\phi(\cdot,-1)\end{matrix}\right)$. With this in mind, we can look at the vector representation of the measure  $A^*Y_t$. We have that
\begin{align*}
\left<\phi, A^*Y_t\right> = \left<A\phi,Y_t\right>
&=\left< (\tfrac{\kappa}{2}\partial_{xx} +\lambda \partial_x)\phi(\cdot, 1),Y_{t,1}\right> + \left<\phi(\cdot,1), \gamma (Y_{t,-1} - Y_{t,1})\right> \\
&\ \ \ \ \ + \left< (\tfrac{\kappa}{2}\partial_{xx} -\lambda \partial_x)\phi(\cdot, -1),Y_{t,-1}\right> + \left<\phi(\cdot,-1), \gamma (Y_{t,1} - Y_{t,-1})\right>\\
&= \left<\phi(\cdot, 1), (\tfrac{\kappa}{2}\partial_{xx} -\lambda \partial_x)Y_{t,1} + \gamma (Y_{t,-1} - Y_{t,1}) \right>\\
&\ \ \ \ \ + \left<\phi(\cdot, -1), (\tfrac{\kappa}{2}\partial_{xx} +\lambda \partial_x)Y_{t,-1} + \gamma (Y_{t,1} - Y_{t,-1}) \right>.
\end{align*}
Therefore $A^*Y_t$ corresponds to the following vector of distributions
\[
A^*Y_t = \left(\begin{matrix}  (\tfrac{\kappa}{2}\partial_{xx} -\lambda \partial_x)Y_{t,1} + \gamma (Y_{t,-1} - Y_{t,1})\\\\(\tfrac{\kappa}{2}\partial_{xx} +\lambda \partial_x)Y_{t,-1} + \gamma (Y_{t,1} - Y_{t,-1}) \end{matrix}\right).
\]
In a similar way we can find a vector representation of the noise part in the SPDE \eqref{SPDE}, namely
\begin{align*}
 \sqrt{2 \kappa\rho}\partial_{x}\dd \mathscr{W}_t + \sqrt{2 \rho\Sigma}\dd \mathscr{\tilde{W}}_t
 &=  \sqrt{2 \kappa\rho}\partial_{x}\left(\begin{matrix} \dd W_{t,1}\\ \dd W_{t,-1} \end{matrix}\right)+ \sqrt{2 \rho\Sigma}\left(\begin{matrix} \dd \tilde{W}_{t,1}\\ \dd\tilde{W}_{t,-1} \end{matrix}\right) \\
 &= \left(\begin{matrix}
 \sqrt{2\kappa \rho} \partial_{x} \dd W_{t,1} + \sqrt{\gamma \rho} \left(\dd\tilde{W}_{t,-1} - \dd\tilde{W}_{t,1}\right)\\
 \sqrt{2\kappa \rho} \partial_{x} \dd W_{t,-1} + \sqrt{\gamma \rho} \left(\dd\tilde{W}_{t,1} - \dd \tilde{W}_{t,-1}\right)
 \end{matrix}\right),
\end{align*}
where all the $\dd W_{t,i}, \dd \tilde{W}_{t,i}$ are  independent space-time white noises  on $\mathbb{R}$.
In this notation, the SPDE in \eqref{SPDE} actually gives us a system of SPDE's given by
\begin{equation*}
    \begin{cases}
    \dd Y_{t,1} = \left[\frac{\kappa}{2} \partial_{xx} Y_{t,1} - \lambda \partial_{x} Y_{t,1} + \gamma \left(Y_{t,-1} - Y_{t,1}\right)\right]\dd t + \sqrt{2\kappa \rho} \partial_{x} \dd W_{t,1} + \sqrt{\gamma \rho} \left(\dd\tilde{W}_{t,-1} - \dd\tilde{W}_{t,1}\right),\\\\
    \dd Y_{t,-1} = \left[\frac{\kappa}{2} \partial_{xx} Y_{t,-1} + \lambda \partial_{x} Y_{t,-1} + \gamma \left(Y_{t,1} - Y_{t,-1}\right)\right]\dd t + \sqrt{2\kappa \rho} \partial_{x} \dd W_{t,-1} + \sqrt{\gamma \rho} \left(\dd\tilde{W}_{t,1} - \dd\tilde{W}_{t,-1}\right).
    \end{cases}
\end{equation*}
Now we are able to sum up these equations to get an SPDE for the fluctuation process of the total density $Z_t$. Just like in the hydrodynamic limit, this will again depend on the difference of the two processes above, denoted by $R_t := Y_{t,1} - Y_{t,-1}$. This gives us the following system of coupled SPDE's
\begin{equation}\label{spde sum}
    \begin{cases}
    \dd Z_t = \left[\frac{\kappa}{2} \partial_{xx} Z_t - \lambda \partial_{x} R_t\right] \dd t +2\sqrt{\kappa \rho} \partial_{x} \dd W_{t,Z},\\\\
    \dd R_t = \left[ \frac{\kappa}{2} \partial_{xx} R_t - \lambda \partial_{x} Z_t -2\gamma R_t\right] \dd t + 2\sqrt{\kappa\rho}\partial_{x} \dd W_{t,R} + 2\sqrt{2\gamma\rho}  \dd\tilde{W}_{t},
    \end{cases}
\end{equation}
where
$$W_{t,Z} = \frac{1}{\sqrt{2}}\left(W_{t,1} + W_{t,-1}\right),\ \ \ \  W_{t,R} = \frac{1}{\sqrt{2}}\left(W_{t,1} - W_{t,-1}\right),\ \ \ \  \tilde{W}_{t} = \frac{1}{\sqrt{2}}\left(\tilde{W}_{t,1} - \dd\tilde{W}_{t,-1}\right),$$
which are all independent space-time white noises on $\mathbb{R}$.

\subsection{Covariance of the total density}
The process $Z_t$  introduced as in \eqref{spde sum} is a (non-Markovian) stationary Gaussian processes. Therefore, we can characterize $Z_t$ through its covariances. Using \eqref{ping} and \eqref{pong}, we can actually relate this covariance to the covariance structure of $Y_t$, which we have already calculated in  Proposition \ref{covariance structure}. In order to do so, for a given  $\phi,\psi \in C_c^\infty$ we define the functions $\bar{\phi}, \bar{\psi}\in C_{c,S}^\infty$ via $\bar{\phi}(x,\sigma) = \phi(x)$ and $\bar{\psi}(x,\sigma) = \psi(x)$. The covariance can then be computed as follows
\begin{align}\label{covar}
  \nn  \E[\left<\phi,Z_t\right> \left<\psi, Z_0\right>]
   \nn &= \E[\left(\left<\phi,Y_{t,1}\right> + \left<\phi,Y_{t,-1}\right>\right)\left(\left<\psi,Y_{0,1}\right> + \left<\psi,Y_{0,-1}\right>\right)]\\
   \nn &=\E[\left(\left<\bar{\phi}(\cdot,1),Y_{t,1}\right> + \left<\bar{\phi}(\cdot,-1),Y_{t,-1}\right>\right)\left(\left<\bar{\psi}(\cdot,1),Y_{0,1}\right> + \left<\bar{\psi}(\cdot,-1),Y_{0,-1}\right>\right)]\\
   \nn &=\E[\left<\bar{\phi},Y_t\right> \left<\bar{\psi},Y_0\right>]\\
    &=\rho\cdot \left<\left< e^{tA}\bar{\phi}, \bar{\psi} \right>\right>.
\end{align}
This covariance strongly resembles the covariance of a stationary Ornstein-Uhlenbeck process, but notice that the semigroup $e^{tA}$ works on the ``extended'' functions $\bar{\phi}, \bar{\psi}$, which corresponds to the non-Markovianity of the
process $\{Z_t, t\geq 0\}$.

Notice that the formula for the covariance obtained in \eqref{covar} is solely based on duality, and is therefore valid as long as we have duality for the multi-layer system, i.e., beyond the case of two internal states and including also interacting cases such as the multi-layer SEP.

\subsection{Closed form equation and large deviations for the case $\kappa=0$}
In the case of $\kappa=0$ the noise term vanishes in the upper equation of \eqref{spde sum} and therefore we  we can solve the system explicitly. Namely, we then find that
\begin{equation*}
    \begin{cases}
    \dd Z_t =  - \lambda \partial_{x} R_t \dd t,  \\\\
    \dd R_t = -\left[ \lambda \partial_{x} Z_t +2\gamma R_t\right] \dd t   + 2\sqrt{\gamma\rho} \dd\tilde{W}_{t}.
    \end{cases}
\end{equation*}
Just like for the hydrodynamic limit of the total density, by now taking a second derivative in time in the first equation we find that $\dd^2 Z_t =  - \lambda \partial_{x} \dd R_t \dd t$. By now filling in $\dd R_t$ from the lower equation, we have that
\begin{align}\label{Zt}
\nn \frac{\dd^2 Z_t}{\dd t^2} &= \lambda^2 \partial_{xx} Z_t -2\gamma\lambda \partial_x R_t    + 2\lambda \sqrt{\gamma\rho}\partial_x \frac{\dd\tilde{W}_{t}}{\dd t}\\
&= \lambda^2 \partial_{xx} Z_t  + 2\gamma \frac{\dd Z_t}{ \dd t} + 2\lambda \sqrt{\gamma\rho}\partial_x \frac{\dd\tilde{W}_{t} }{\dd t}.
\end{align}

From the expression above we are also able to obtain the rate function for  the large deviations of $Z_{t}^{(\varepsilon)}$ in the small noise regime, i.e., where we add a factor $\varepsilon$ before the noise $\tilde{W}_{t}$ which we will send to zero. I.e., we are interested in the large deviations of Schilder type (see
\cite{DS}, \cite{lifshits}) for the family of Gaussian process  given by
\begin{equation}\label{Zte}
\frac{\dd^2 Z_{t}^{(\varepsilon)}}{\dd t^2} = \lambda^2 \partial_{xx} Z_{t}^{(\varepsilon)}  + 2\gamma \frac{\dd Z_{t}^{(\varepsilon)}}{ \dd t} +  \varepsilon 2\lambda \sqrt{\gamma\rho}\partial_x \frac{\dd\tilde{W}_{t} }{\dd t}.
\end{equation}
We use that
\[
\p\left( \varepsilon \partial_x \frac{\dd\tilde{W}_{t}}{\dd t} \asymp \Gamma(t,x) \right) \asymp \exp\left(-\varepsilon^{-2} \frac{1}{2}\int_0^T ||\Gamma(t,\cdot)||^2_{H_{-1}}\dd t\right),
\]
which has to be interpreted in the sense of the large deviation principle in the space of space-time distributions. The rate function in the above equation  can be derived from  the log-moment-generating function of a space-time white noise on $\mathbb{R}$, which for a test function $\phi \in C_c^\infty( [0,T]\times \mathbb{R})$ is equal to
\[
\Lambda(\phi) = \lim_{\varepsilon \to 0} \varepsilon^2 \log\left(\E[e^{\varepsilon^{-1}\left<\phi,  \partial_x \frac{\dd W_t}{\dd t}\right>}]\right) =  \frac{1}{2} \left<\partial_x\phi,\partial_x\phi\right>_{L^2(\mathbb{R}\times [0,T])}.
\]
The Legendre transform of $\Lambda$ then yields  the rate function,
\begin{equation*}
\begin{split}
\Lambda^*(\Gamma(t,x) )  &= \sup_{\phi \in C_c^\infty( [0,T]\times \mathbb{R})} \left\{\left<\phi, \Gamma  \right>_{L^2([0,T]\times \mathbb{R})}  - \frac{1}{2}  \left<\partial_x \phi, \partial_x\phi\right>_{L^2( [0,T]\times \mathbb{R})} \right\}\\
&= \frac{1}{2}\int_0^T  \left|\left|\Gamma(t,\cdot)\right|\right|^2_{H^{-1}} \dd t.
\end{split}
\end{equation*}
As a consequence, we obtain the large deviation principle for the random space-time distribution $Z_{t}^{(\varepsilon)}$, namely from $\eqref{Zte}$ it follows that
\begin{equation}\label{di}
\begin{split}
\p\left( Z_{t}^{(\varepsilon)} \asymp \Gamma(t,x)\right) &= \p\left( \varepsilon \partial_x\frac{\dd\tilde{W_t}}{\dd t} \asymp \frac{1}{2\lambda\sqrt{\gamma\rho}} \left( \ddot{\Gamma}(t,x) -2\gamma \dot{\Gamma}(t,x) - \lambda^2 \partial_{xx}\Gamma(t,x)\right) \right)\\
&\asymp \exp\left(-\varepsilon^{-2} \frac{1}{4\lambda\sqrt{\gamma\rho}} \int_0^T \left|\left|  \ddot{\Gamma}(t,\cdot) -2\gamma \dot{\Gamma}(t,\cdot) - \lambda^2 \partial_{xx}\Gamma(t,\cdot)\right|\right|^2_{H_{-1}}\dd t\right).
\end{split}
\end{equation}

\section{Proof of Theorem \ref{fluctuations}}\label{section 4}

In this section we prove Theorem \ref{fluctuations}, following the line of proof of Van Ginkel and Redig in \cite{Ginkel}.
For the readers convenience we sketch the main steps.

We start by introducing the Dynkin martingales of $Y_t^N(\phi)$. For every $\phi \in C_{c,S}^\infty$ and $N\in\mathbb{N}$, let  $\{\mathscr{F}^{N}_{t}:t\geq0\}$ be the filtration generated by $\{Y^N_{t}:t\geq 0\}$. Because the configuration process $\{\eta_t:t\geq 0\}$ is a  Markov processes
the following processes
\begin{equation}\label{dynk mart}
\begin{split}
    \mathscr{M}^{N,\phi}_{t}(Y^N) &= Y^N_{t}(\phi) - Y^N_{0}(\phi) - \int_0^t L_N Y^N_{s}(\phi)\dd s,\\
    \mathscr{N}^{N,\phi}_{t}(Y^N) &=\mathscr{M}^{N,\phi}_{t}(Y^N)^2 - \int_0^t \Gamma_{s}^{N,\phi}(Y^N)\dd s,
\end{split}
\end{equation}
are $\mathscr{F}^{N}_{t}$-martingales,  where $\Gamma_{s}^{N,\phi}$ is the so-called Carr\'e du champ operator given by
\begin{equation}\label{Gamma}
    \Gamma_{s}^{N,\phi}(Y^N) := L_N\big(Y_{s}^N(\phi)^2\big) - 2Y_{s}^N(\phi)L_NY_{s}^N(\phi) .
\end{equation}
The aim is then to prove that as $N\to\infty$, the martingales in \eqref{dynk mart} converge to the martingales from \eqref{mart prob}. This fact, complemented with a proof of tightness and the fact that the martingale problem \eqref{mart prob} has a unique solution, then completes the proof.
In Section \ref{substituting the martingales} we prove the convergence of the martingales, in Section \ref{Tigh} we prove the tightness, and in Section \ref{Uniq} we prove the uniqueness of the solution of the martingale problem \eqref{mart prob}.
\subsection{Substituting the martingales}\label{substituting the martingales}
Our goal for this section is to show that in the limit as $N\to \infty$, we can substitute $\mathscr{M}^{\phi}_t(Y^N)$ and $\mathscr{N}^{\phi}_t(Y^N)$ (with $\mathscr{M}^{\phi}_t$ and $\mathscr{N}^{\phi}_t$ defined as in \eqref{mart prob}) for $\scrM^{N,\phi}_t(Y^N)$ and $\scrN^{N,\phi}_t(Y^N)$ respectively. We do so in the  Propositions \ref{substitute M} and \ref{substitute N}. We recall the reader that the expectation $\E$ stands for the stationary expectation starting from the initial configuration distributed according to $\eta_0\sim \mu_\rho$.
\bp\label{substitute M}
For all $\phi \in C_{c,S}^\infty$ we have
\beq
\nonumber\lim_{N\to\infty} \E\left[\left|\scrM^{N,\phi}_t(Y^N) -\mathscr{M}^{\phi}_t(Y^N) \right|^2\right]=0.
\eeq
\ep
\begin{proof}
First of all, note that by definition
\begin{equation*}
    \E\left[\left|\scrM^{N,\phi}_t(Y^N) -\mathscr{M}^{\phi}_t(Y^N) \right|^2\right]=  \E\left[\left|\int_0^t L_NY_{s}^N(\phi)\dd s - \int_0^tY^N_s( A\phi) \dd s\right|^2\right].
\end{equation*}
For a given $(x,\sigma) \in V$ we have that
\begin{equation*}
\begin{split}
L_N \eta(x,\sigma) &= \kappa N^2[\eta(x+1,\sigma)+\eta(x-1,\sigma) - 2\eta(x,\sigma)]\\
& \ \ \ + \lambda N[\eta(x-\sigma,\sigma)-\eta(x,\sigma)]\\
& \ \ \ + \sum_{\sigma'\in S} c(\sigma,\sigma') [\eta(x,\sigma') - \eta(x,\sigma)],
\end{split}
\end{equation*}
and so in particular we find that
\begin{equation*}
\begin{split}
L_N Y_s^N(\phi) &= \frac{1}{\sqrt{N}} \sum_{(x,\sigma) \in V}\big( L_N\eta_s(x,\sigma)\big)  \phi(\tfrac{x}{N},\sigma) = \frac{1}{\sqrt{N}} \sum_{(x,\sigma) \in V} \eta_s(x,\sigma)\cdot  (\mathscr{L}_N\phi)(\tfrac{x}{N},\sigma),
\end{split}
\end{equation*}
 where we remind the reader that $\mathscr{L}_N$ is the generator of a single run-and-tumble particle on the rescaled space $\frac{1}{N}\mathbb{Z}\times S$.
Now, using that for any $\phi\in C_{c,S}^\infty$ we have that
\[
\sum_{(x,\sigma)\in V} \rho \cdot (\mathscr{L}_N\phi)(\tfrac{x}{N},\sigma) = 0,
\]
we are able to write
\[
L_N Y_s^N(\phi) = \frac{1}{\sqrt{N}} \sum_{(x,\sigma) \in V} (\eta_s(x,\sigma)-\rho) \cdot  (\mathscr{L}_N\phi)(\tfrac{x}{N},\sigma).
\]
Since $\mathscr{L}_N\phi \to A\phi$ uniformly, where $A$ is defined in \eqref{A}, we have that
\begin{equation}\label{LN sub}
\begin{split}
    L_NY_{s}^N(\phi) &= \frac{1}{\sqrt{N}} \sum_{(x,\sigma)\in V} (\eta_s(x,\sigma)-\rho) \cdot (A\phi)(\tfrac{x}{N},\sigma) + R_1(\phi,N, s),
    \end{split}
\end{equation}
where $R_1(\phi,N, s)$ is an error term produced by  the Taylor approximations. Since $\phi$ is compactly supported,  if we define $V_{\phi}^N := \{(x,\sigma) \in V, \phi(\tfrac{x}{N},\sigma)\neq 0\}$ then $|V_\phi^N| = \mathcal{O}(N)$. Furthermore, the error term is bounded in the following way
\begin{equation}\label{R1}
    |R_1(\phi,N, s)| \leq \frac{1}{N^{3/2}}\sum_{(x,\sigma)\in V_{\phi}^N}  (\eta_{s}(x,\sigma)-\rho)(\kappa ||\partial_{xxx}\phi||_\infty + \lambda \sigma^2||\partial_{xx}\phi||_\infty).
\end{equation}
Therefore we find that for every $\phi \in C_{c,S}^\infty$ and $t\geq0$,
\begin{equation}\label{vanishing R1}
\begin{split}
    \E\left[R_1(\phi,N,s)^2\right] & \leq  \frac{1}{N^3} \E\left[\sum_{(x,\sigma),(y,\sigma') \in V_{\phi}^N}  (\eta_s(x,\sigma) - \rho)(\eta_s(y,\sigma')-\rho) (\kappa||\partial_{xxx}\phi||_\infty + \lambda\sigma^2||\partial_{xx}\phi||_\infty)^2\right]\\
    &=\frac{1}{N^3}\sum_{(x,\sigma),(y,\sigma') \in V_{\phi}^N}  \Cov\big(\eta_s(x,\sigma), \eta_s(y,\sigma')\big) (\kappa||\partial_{xxx}\phi||_\infty + \lambda \sigma^2||\partial_{xx}\phi||_\infty)^2.\\
\end{split}
\end{equation}
Since we are starting  the process $\eta_t$ from the invariant product measure $\mu_\rho$, we have that
\begin{equation}\label{Cov}
    \Cov\big(\eta_s(x,\sigma), \eta_s(y,\sigma')\big) = \rho\cdot I\big((x,\sigma) = (y,\sigma')\big).
\end{equation}
Therefore,
\[
\E\left[R_1(\phi,N,s)^2\right]  \leq \frac{1}{N^3} |V_{\phi}^N| \rho(\kappa||\partial_{xxx}\phi||_\infty + \lambda\sigma^2||\partial_{xx}\phi||_\infty)^2\to 0,
\]
where we used the fact that $|V_{\phi}^N| = \mathcal{O}(N)$. Note that the above convergence is uniform in $s$, and therefore by dominated convergence we find that
\begin{align*}
    \lim_{N\to\infty} \E\left[\left|\scrM^{N,\phi}_t(Y^N) -\mathscr{M}^{\phi}_t(Y^N) \right|^2\right]
    = \lim_{N\to\infty}  \int_0^t \E\left[R_1(\phi,N,s)^2\right] \dd s=0,
\end{align*}
which concludes the proof.
\end{proof}
The substitution of $\scrN_t^\phi(Y^N)$ is a bit more work and requires a fourth moment estimate. We start by  proving two lemmas. The proof of the substitution result in Proposition \ref{substitute N} immediately follows from these lemmas.

\bl\label{substitute M2}
For all $\phi \in C_{c,S}^\infty$ we have the following
\begin{equation}\label{lim 1}
    \lim_{k\to\infty} \E\left[\left(\scrM_{t}^{N, \phi}(Y^{N})^2 - \scrM_{t}^{\phi}(Y^{N})^2\right)^2\right]=0.
\end{equation}
\el
\begin{proof}
We start with the following application of H\"older's inequality
\begin{equation}\label{M2}
\begin{split}
    \E\left[\left(\scrM_{t}^{N, \phi}(Y^{N})^2 - \scrM_{t}^{\phi}(Y^{N})^2\right)^2\right]
        &=
    \E\left[\left(\scrM_{t}^{N, \phi}(Y^{N}) - \scrM_{t}^{\phi}(Y^{N})\right)^2\left(\scrM_{t}^{N, \phi}(Y^{N})+ \scrM_{t}^{\phi}(Y^{N})\right)^2\right]\\
        &\leq\left(\E\left[\left(\scrM_{t}^{N, \phi}(Y^{N}) - \scrM_{t}^{\phi}(Y^{N})\right)^4\right]\cdot \E\left[\left(\scrM_{t}^{N, \phi}(Y^{N})+ \scrM_{t}^\phi(Y^{N})\right)^4\right]\right)^{\tfrac{1}{2}}.
        \end{split}
\end{equation}
We will first show that the first expectation in the last line vanishes as $N\to \infty$, and afterwards we will show that the second expectation is uniformly bounded in $N$.  Note that by \eqref{LN sub}
\begin{align*}
    \E\left[\left(\scrM_{t}^{N, \phi}(Y^{N}) -  \scrM_{t}^{\phi}(Y^{N})\right)^4\right]= \E\left[ \left(\int_0^t \left[R_1(\phi,N, s)\right]\dd s\right)^4\right]
    \leq t^3 \int_0^T \E\left[R_1(\phi,N, s)^4\right]\dd s.
\end{align*}
 Using the bound in  $\eqref{R1}$ we find that
\begin{equation*}
    \E\left[R_1(\phi,N, s)^4\right] \leq \frac{1}{N^6} \sum_{\substack{(x_i,\sigma_i) \in V_{\phi}^N\\1\leq i\leq 4}} \E\left[\prod_{i=1}^4(\eta_{s}(x_i,\sigma_i) - \rho)\right](\kappa||\partial_{xxx}\phi||_\infty + \lambda\sigma^2||\partial_{xx}\phi||_\infty)^4.
\end{equation*}
Since we start from the product Poisson measure $\mu_\rho$, we only get non-zero contributions in the expectation on the right-hand side when all $(x_i,\sigma_i)$ are equal or when we have two distinct pairs, given by
\begin{equation*}
    \E\left[ (\eta_s(x,\sigma)-\rho)^4\right] = 3\rho^2 + \rho, \ \ \ \ \E\left[(\eta_s(x,\sigma)-\rho)^2(\eta_s(y,\sigma')-\rho)^2\right] = \rho^2.
\end{equation*}
Therefore, it follows that
\begin{equation}\label{vanish R1}
  \E\left[R_1(\phi,N, s)^4\right]  \leq \frac{1}{N^6}\left(|V_{\phi}^N| (3\rho^2 + \rho) + |V_{\phi}^N|^2\rho^2\right)(\kappa||\partial_{xxx}\phi||_\infty + \lambda\sigma^2||\partial_{xx}\phi||_\infty)^4,
\end{equation}
and so $R_1(\phi, N,s,\sigma) \xrightarrow{L^4} 0$ uniformly in $s$. From this we can conclude that
\begin{equation*}
    \E\left[\left(\scrM_{t}^{N, \phi}(Y^{N}) -  \scrM_{t}^{\phi}(Y^{N})\right)^4\right] \leq t^3 \int_0^t \E\left[R_1(\phi,N, s)^4\right]\dd s \to 0.
\end{equation*}

To now show that the second expectation in the last line of \eqref{M2} is uniformly bounded in $N$, note that
\begin{equation}\label{sum}
    \E\left[\left(\scrM_{t}^{N, \phi}(Y^{N})+ \scrM_{t}^\phi(Y^{N})\right)^4\right] \leq 8\left(\E\left[\left(\scrM_{t}^{N, \phi}(Y^{N})\right)^4\right] + \E\left[\left(\scrM_{t}^{ \phi}(Y^{N})\right)^4\right]\right),
\end{equation}
and similarly
\begin{equation}\label{3 expectations}
    \begin{split}
        \E\left[\left(\scrM_{t}^{ \phi}(Y^{N})\right)^4\right] \leq 27\left(\E\left[Y_{t}^N(\phi)^4\right] + \E\left[Y_{0}^N(\phi)^4\right] + \E\left[\left(\int_0^t Y_{s}^N(A\phi)\dd s\right)^4\right]\right).
    \end{split}
\end{equation}
Now we  need to show that three expectations on the right-hand-side  are uniformly bounded. For the first expectation, we find that
\begin{equation*}
\begin{split}
    \E\left[Y_{t}^N(\phi)^4\right] &\leq \frac{1}{N^2} \cdot  \sum_{(x_1,\sigma_1)\in V_{\phi}^N} \cdots \sum_{(x_4,\sigma_4)\in V_{\phi}^N}    \E\left[\prod_{i=1}^4(\eta_t(x_i,\sigma_i) - \rho)\right]||\phi||_\infty.
\end{split}
\end{equation*}
Similarly as in \eqref{vanish R1}, we find that
\begin{equation}\label{uniform Y}
\E\left[Y_{t}^N(\phi)^4\right] \leq \frac{1}{N^2} \left(|V_{\phi}^N| (3\rho^2 + \rho) + |V_{\phi}^N|^2\rho^2\right)||\phi||_\infty = \mathcal{O}(1),
\end{equation}
hence it is  uniformly bounded, and similar approaches can be used for $ \E\left[Y_{0}^N(\phi)^4\right]$ and $ \E\left[Y_{s}^N(A\phi)^4\right]$. The fact that the last expectation in \eqref{3 expectations} is uniformly bounded now follows from an application of H\"older's inequality, namely
\begin{equation*}
    \E\left[\left(\int_0^t Y_{s}^N(A\phi)\dd s\right)^4\right] \leq t^3 \int_0^T \E\left[\left(Y_{s}^N(A\phi)\right)^4\right]\dd s.
\end{equation*}
Therefore we know that  $\E\left[\big(\scrM_{t}^{ \phi}(Y^{N})\big)^4\right]$ is uniformly bounded. The proof for $\E\left[\big(\scrM_{t}^{N, \phi}(Y^{N})\big)^4\right]$ works the same way if we  use that
\begin{equation*}
    \begin{split}
         \E\left[\left(L_NY_{s}^N(\phi)\right)^4\right] =8\left(\E\left[\left(Y_{s}^N(A\phi)\right)^4\right] + \E\left[R_1(\phi,N,t,\sigma)^4\right]\right),
    \end{split}
\end{equation*}
where by \eqref{vanish R1} we already know that $\E\left[R_1(\phi,N,t,\sigma)^4\right]$ is uniformly bounded.  Hence we can conclude that \eqref{lim 1} holds.
\end{proof}

\bl\label{substitute Gamma} For all $\phi \in C_{c,S}^\infty$ we have the following
\begin{equation*}
    \lim_{N\to\infty} \E\left[\left(\int_0^t   \Gamma_{s}^{N,\phi}(Y^N)\dd s - 2t\kappa \rho  \left<\left<\partial_{x} \phi,\partial_{x} \phi\right>\right>- 2t\rho \left<\left<\phi, \Sigma\phi\right>\right>\right)^2\right] = 0,
\end{equation*}
with $\Sigma$ defined as in \eqref{Sigma}.
\el
\begin{proof}
First we recall that for a Markov process with generator $L$ determined by the transition rates $r(\eta,\eta')$ the carr\'{e} du champ operator is computed as follows.
\begin{equation}\label{deriv}
    \begin{split}
        Lf^2(\eta) - 2f(\eta)\cdot L f(\eta)
        &= \sum_{\eta' \in \Om} r(\eta,\eta')\Big( \big( f^2(\eta') - f^2(\eta)\big) - 2\big(f(\eta)f(\eta') - f^2(\eta)\big)\Big)\\
        &= \sum_{\eta'\in\Om} r(\eta,\eta') \big(f(\eta') - f(\eta)\big)^2,
    \end{split}
\end{equation}
Translating this to our setting with $L=L_N$ and $f= Y^N$ we obtain
\ben
    \begin{split}
          \Gamma_{s}^{N,\phi}(Y^N)
        &=\kappa N \sum_{(x,\sigma)\in V} \eta_s(x,\sigma)\left(\phi(\tfrac{x+1}{N},\sigma) - \phi(\tfrac{x}{N},\sigma))^2 + (\phi(\tfrac{x-1}{N},\sigma) - \phi(\tfrac{x}{N},\sigma))^2\right)\\
        &\ \ \ + \lambda  \sum_{(x,\sigma)\in V} \eta_s(x,\sigma) \left(\phi(\tfrac{x+\sigma}{N},\sigma) - \phi(\tfrac{x}{N},\sigma)\right)^2\\
        &\ \ \ + \frac{1 }{N} \sum_{(x,\sigma)\in V} \sum_{\sigma' \in S} c(\sigma,\sigma')\eta_s(x,\sigma)\big(\phi(x,\sigma')-\phi(x,\sigma))^2.
    \end{split}
\een
Using Taylor expansion with rest term, we can write
\begin{equation}\label{gamma 1 deriv}
    \begin{split}
          \Gamma_{s}^{N,\phi}(Y^N)
         &= \frac{2\kappa}{N} \sum_{(x,\sigma)\in V}\eta_s(x,\sigma) \left(\partial_{x}\phi(\tfrac{x}{N},\sigma)\right)^2  +\frac{1}{N} \sum_{(x,\sigma)\in V}\sum_{\sigma'\in S} c(\sigma,\sigma') \eta_s(x,\sigma)(\phi(\tfrac{x}{N},\sigma') - \phi(\tfrac{x}{N},\sigma)^2\\
         &\ \ \ \ \ \ \ \  + R_2(\phi,s,N),
    \end{split}
\end{equation}
with $R_2(\phi,s,N)$ the error term, which is bounded as follows
\ben
|R_2(\phi,s,N)| \leq \kappa\frac{1}{N^3}\sum_{(x,\sigma)\in V_\phi^N} \eta_s(x,\sigma)\kappa||\partial_{xx}\phi||_\infty + \frac{1}{N^2}\sum_{(x,\sigma)\in V_\phi^N} \eta_s(x,\sigma) \lambda \sigma ||\phi'||_\infty.
\een
Following the line of thought leading to \eqref{vanishing R1}, we obtain that $R_2(\phi,s,N) \xrightarrow{L^2}0$. Therefore, for the expectation we find that
\begin{equation}\label{exp gamma 1}
    \begin{split}
    \E\left[\Gamma_{s}^{N,\phi}(Y^N)\right] &= \frac{2\kappa\rho}{N} \sum_{(x,\sigma)\in V} (\partial_{x} \phi(\tfrac{x}{N},\sigma))^2 + \frac{2\rho}{N} \sum_{\sigma'\in S}c(\sigma,\sigma')(\phi(\tfrac{x}{N},\sigma') - \phi(\tfrac{x}{N},\sigma))^2 + \E\left[R_2(\phi,s,N)\right]\\
    &\to 2\kappa\rho \left<\left<\partial_x \phi, \partial_x\phi\right>\right> + 2 \rho \left<\left<\phi, \Sigma \phi\right>\right>,
    \end{split}
\end{equation}
and for the variance
\begin{equation} \label{var gamma 1}
\begin{split}
\mathrm{Var}\left[  \Gamma_{s}^{N,\phi}(Y^N)\right]
&\leq \frac{C(\phi,s)}{N^2}\sum_{(x,\sigma), (y,\sigma')\in V_{\phi}^N} \Cov\big(\eta_s(x,\sigma), \eta_s(y,\sigma')\big)\\
&= \frac{C(\phi,s)}{N^2} |V_{\phi}^N|\rho \to 0,
\end{split}
\end{equation}
with $C(\phi,s)$ some constant and where we have used \eqref{Cov} for the equality. Since the variance converges to zero, this means that $ \Gamma_{s}^{N,\phi}(Y^N)$ converges to its mean in $L^2$. Therefore
\begin{align*}
 \lim_{N\to\infty} &\E\left[\left(\int_0^t   \Gamma_{s}^{N,\phi}(Y^N)\dd s - 2t\kappa \rho  \left<\left<\partial_{x} \phi,\partial_{x} \phi\right>\right>- 2t\rho \left<\left<\phi, \Sigma\phi\right>\right>\right)^2\right]\\
 &\leq \lim_{N\to\infty} \int_0^t   \E\left[\left(\Gamma_{s}^{N,\phi}(Y^N) - 2\kappa \rho  \left<\left<\partial_{x} \phi,\partial_{x} \phi\right>\right>- 2\rho \left<\left<\phi, \Sigma\phi\right>\right>\right)^2\right] \dd s\\
 &=0,
\end{align*}
where we used dominated convergence for the last equality.
\end{proof}

\bp\label{substitute N}
For all $\phi \in C_{c,S}^\infty$
$$\lim_{N\to\infty}\E\left[\left|\scrN^{N,\phi}_{t}(Y^{N}) - \scrN^{\phi}_{t}(Y^{N}) \right|^2\right]=0.$$
\ep
\begin{proof}
We have that
\begin{align*}
    \hspace{-5mm}\E\left[\left|\scrN^{N,\phi}_{t}(Y^{N}) - \scrN^{\phi}_{t}(Y^{N}) \right|^2\right]
    &\leq 2\E\left[\left(\scrM_{t}^{N, \phi}(Y^{N})^2 - \scrM_{t}^{\phi}(Y^{N})^2\right)^2\right]\\
    & \ \ \ \  + 2\E\left[\left(\int_0^t   \Gamma_{s}^{N,\phi}(Y^N)\dd s - 2t\kappa \rho  \left<\left<\partial_{x} \phi,\partial_{x} \phi\right>\right>- 2t\rho \left<\left<\phi, \Sigma\phi\right>\right>\right)\right].
\end{align*}
The proof now follows from Lemma   \ref{substitute M2} and \ref{substitute Gamma}.
\end{proof}

\subsection{Tightness}\label{Tigh}
In this section we will show the tightness of the collection $\{Y^N :  N \in \mathbb{N}\}$.
\begin{proposition}\label{tight}
$\{Y^N :  N \in \mathbb{N}\}$ is tight in $D([0,T];(C_{c,S}^\infty)^*)$.
\end{proposition}
\begin{proof}
Since $C_{c,S}^{\infty}$ is a nuclear space, by Mitoma \cite[Theorem 4.1]{Mitoma} it suffices to prove that for a fixed $\phi \in C_{c,S}^{\infty}$ we have that $\{Y^N(\phi) : N\in\mathbb{N}\}$ is tight in the path space $D([0,T]; \mathbb{R})$. Aldous' criterion, as stated in \cite[Theorem 1]{Aldous}, tells us that it suffices to show the following two things:
\begin{enumerate}[label=\textbf{A.\arabic*}]
    \item \label{1}For all $t\in[0,T]$ and $\varepsilon>0$ there exists a compact $K(t,\varepsilon) \in \mathbb{R}$ such that
    \begin{equation*}
        \sup_{N\in \mathbb{N}} \mathbb{P}\big(Y_t^N(\phi) \notin K(t,\varepsilon)\big) \leq \varepsilon.
    \end{equation*}
    \item \label{2} For all $\varepsilon>0$
    \begin{equation*}
        \lim_{\delta \to 0} \limsup_{N\to\infty} \sup_{\substack{\uptau \in \mathscr{T}_T\\ \theta \leq \delta}} \mathbb{P}\big(|Y_\uptau^N(\phi) - Y_{\uptau + \theta}^N(\phi)|>\varepsilon\big) = 0,
    \end{equation*}
    with $\mathscr{T}_T$ the set of all stopping times bounded by $T$.
\end{enumerate}

Fix $t \in [0,T]$ and $\phi \in C_{c,S}^\infty$. Then, for every $\sigma \in S$ we have that
\begin{equation*}
    \begin{split}
        &\mathbb{E}[Y_{t}^N(\phi)] = \frac{1}{\sqrt{N}} \sum_{(x,\sigma)\in V}\mathbb{E}\left[  \eta_t(x,\sigma)-\rho\right]\phi(\tfrac{x}{N},\sigma)=0,\\
        &\mathrm{Var}[Y_{t}^N(\phi)] = \frac{1}{\sqrt{N}} \sum_{(x,\sigma)\in V}  \Var\left[\eta_t(x,\sigma)-\rho\right]\phi(\tfrac{x}{N},\sigma)= \frac{1}{N}\rho  \sum_{(x,\sigma)\in V} \phi^2(\tfrac{x}{N},\sigma).
    \end{split}
\end{equation*}
By the central limit theorem, we therefore see that every $Y_{t}^N(\phi)$ converges in distribution to the  normal distribution $\mathcal{N}\big(0, \rho \left<\left<\phi,\phi\right>\right>\big)$. This implies the tightness of the real-valued random variables $\{Y_t^N(\phi) : N \in\mathbb{N}\}$, and therefore also \ref{1}.\\

To prove \ref{2}, we note that for every bounded stopping time $\uptau \in \mathscr{T}_T$ we have that
\begin{equation*}
    Y_{\uptau}^N(\phi) = \scrM^{N,\phi}_{\uptau}(Y^N) + Y_{0}^N(\phi) + \int_0^\uptau L_NY_{s}^N( \phi)ds,
\end{equation*}
with $\scrM^{N,\phi}_{\uptau}(Y^N)$ the Dynkin martingale of $Y_{\uptau}^N(\phi)$. Using the Markov inequality, we can then deduce that
\begin{equation}\label{prob fluctuation}
    \begin{split}
         \mathbb{P}\big(|Y_{\uptau}^N(\phi) - Y_{\uptau+\theta}^N(\phi)|>\varepsilon\big) &\leq \frac{1}{\varepsilon^2}\E\left[\left(Y_{\uptau}^N(\phi) - Y_{\uptau+\theta,\sigma}^N(\phi)\right)^2\right]\\
         &\leq\frac{2}{\varepsilon^2}\left(\E\left[\left(\scrM^{N,\phi}_{\uptau}(Y^N) - \scrM_{\uptau+\theta}^{N,\phi}(Y^N)\right)^2\right] + \E\left[\left(\int_\uptau^{\uptau +\theta} L_NY_{s}^N(\phi) \dd s\right)^2\right]\right).
    \end{split}
\end{equation}

For the integral term, note that by the Cauchy-Schwarz inequality and Fubini we have that
\begin{equation}\label{integral term}
    \begin{split}
        \E\left[\left( \int_\uptau^{\uptau+\theta} L_NY_{s}^N (\phi)dr\right)^2 \right] &\leq \sqrt{\theta} \cdot\left( \E\left[\int_0^{T+\theta} \left(L_NY_{s}^N (\phi)\right)^2 \dd s\right]\right)^{\frac{1}{2}}\\
        & = \sqrt{\theta}\cdot \left(\int_0^{T+\theta} \E\left[ \left(L_NY_{s}^N(\phi)\right)^2\right]\dd s\right)^\frac{1}{2}.\\
    \end{split}
\end{equation}
In the proof of Lemma \ref{substitute M2} we have shown that  $\{L_NY_{s}^N(\phi):N\in\mathbb{N}\}$ is uniformly bounded in $L^4$, hence it is also uniformly bounded in $L^2$, i.e.
\begin{equation}\label{integral sup}
    C:=\sup_{N\in \mathbb{N}}\E\left[ \left(L_NY_{s}^N(\phi)\right)^2\right] <\infty.
\end{equation}
Combining (\ref{integral term}) and (\ref{integral sup}), we find that
\begin{equation}\label{integral limit}
    \lim_{\delta\to0} \limsup_{N\to\infty} \sup_{\substack{\uptau \in \mathscr{T}_T \\ \theta\leq \delta}}  \E\left[\left( \int_\uptau^{\uptau+\theta} L_NY_{s}^N (\phi)dr\right)^2 \right]  \leq \lim_{\delta \to 0} \sqrt{\delta CT} = 0.
\end{equation}

For the martingale, by the martingale property we have that  $$\E\left[\scrM_{\uptau}^{N,\phi}(Y^N)\scrM_{\uptau+\theta}^{N,\phi}(Y^N)\right]=\E\left[\left(\scrM_{\uptau}^{N,\phi}(Y^N)\right)^2\right],$$
hence we see that
\begin{equation*}
    \begin{split}
            \E\left[\left(\scrM^{N,\phi}_{\uptau}(Y^N) -\scrM_{\uptau+\theta}^{N,\phi}(Y^N)\right)^2\right] =\E\left[\left(\scrM_{\uptau+\theta}^{N,\phi}(Y^N)\right)^2 - \Big(\scrM_{\uptau}^{N,\phi}(Y^N)\Big)^2\right].
    \end{split}
\end{equation*}
Since $\E\left[\scrM_{0}^{N,\phi}(Y^N)\right] = 0$, we can use that
\begin{equation*}
    \E\left[\left(\scrM_{t}^{N,\phi}(Y^N)\right)^2\right] = \E\left[\int_0^t\Gamma_{s}^{N,\phi}(Y^N)\right]\dd s,
\end{equation*}
because $\int_0^t\Gamma_{s}^{N,\phi}(Y^N)\dd s$ is the quadratic variation of the process $\scrM_{t}^{N,\phi}(Y^N)$. Furthermore,  $\E\left[\left(\Gamma_{s}^{N,\phi}(Y^N)\right)^2\right]$ is uniformly bounded since $\Gamma_{s}^{N,\phi}(Y^N)$ converges in $L^2$, hence
\begin{equation*}
    \begin{split}
           \sup_{N\in\mathbb{N}} \E\left[\left(\scrM^{N,\phi}_{\uptau}(Y^N) -\scrM_{\uptau+\theta}^{N,\phi}(Y^N)\right)^2\right]
            &= \sup_{N\in\mathbb{N}} \E\left[\int_\uptau^{\uptau+\theta} \Gamma_{s}^{N,\phi}(Y^N)\right]\dd s, \\
            &\leq \sqrt{\theta} \cdot \left(  \int_0^{T+\theta} \sup_{N\in\mathbb{N}} \E\left[\left(\Gamma_{s}^{N,\phi}(Y^N)\right)^2\right]\dd s\right)^{\frac{1}{2}}<\infty,
                \end{split}
\end{equation*}
where we used Cauchy Schwarz in the second line. From this we can again conclude that
\begin{equation}\label{martingale limit}
     \lim_{\delta\to0} \limsup_{N\to\infty} \sup_{\substack{\uptau \in \mathscr{T}_T \\ \theta\leq \delta}}  \E\left[\left(\scrM^{N,\phi}_{\uptau}(Y^N) -\scrM_{\uptau+\theta}^{N,\phi}(Y^N)\right)^2\right]  = 0.
\end{equation}
Combining (\ref{integral limit}) and (\ref{martingale limit}) with (\ref{prob fluctuation}), we indeed find that (\ref{2}) holds.
\end{proof}

\subsection{Uniqueness of limits}\label{Uniq}
By the tightness, there exists a subsequence $N_k$ and a process $Y\in D([0,T]; (C_{c,S}^\infty)^*) $ such that $Y^{N_k}\to Y$ in distribution.
\begin{lemma}\label{continuity}
For each $\phi \in C_{c,S}^\infty$ we have that $t\mapsto Y_t(\phi)$ is a.s. continuous.
\end{lemma}
\begin{proof}
We define the following functions
\begin{equation*}
    w_\delta(X) = \sup_{|t-s|<\delta}|X_t-X_s|, \ \ \ \ \ w_\delta'(X) = \inf_{\substack{0=t_0<t_1<...<t_r=1\\t_i - t_{i-1}<\delta}}\ \max_{1\leq i \leq r}\  \sup_{t_{i-1}\leq s < t \leq t_i} |X_t - X_s|,
\end{equation*}
then we have the following inequality
\begin{equation}\label{w and w'}
    w_\delta(X) \leq 2w_\delta'(X) + \sup_{t} |X_t - X_{t^-}|.
\end{equation}
From \ref{2} it follows for all $\varepsilon>0$ and all $\sigma \in S$ we have that
\begin{equation}\label{bla}
    \lim_{\delta\to0}\limsup_{N\to\infty} \p(w'_\delta(Y^N(\phi))\geq \varepsilon)=0.
\end{equation}
Now note that
\begin{equation} \label{bli}
    \sup_{t} \left|Y_{t}^N(\phi) - Y_{t^-}^N(\phi)\right| \leq \sup_t \frac{1}{\sqrt{N}} \sum_{v \in V} \left|(\eta_t(v) - \eta_{t^-}(v))\phi(v)\right|\leq \frac{1}{\sqrt{N}} ||\phi||_\infty\to 0,
\end{equation}
where we used that there can be at most one jump between the times $t$ and $t^-$ for the second inequality.
Therefore, by combining \eqref{bla} and \eqref{bli} with \eqref{w and w'} we can conclude that
\begin{equation*}
    \lim_{\delta\to0}\limsup_{N\to\infty} \p(w_\delta(Y^N(\phi))\geq \varepsilon)=0.
\end{equation*}
Therefore we find that $t\mapsto Y_t(\phi)$ is a.s. continuous.
\end{proof}

Finally we show that $Y$ solves the martingale problem in \eqref{mart prob}.
\begin{proposition}
For every $\phi \in C_{c,S}^\infty$ the processes $\scrM_t^{\phi}(Y)$ and $\scrN_t^{\phi}(Y)$ defined in \eqref{mart prob} are martingales with respect to the filtration $\{\mathscr{F}_t: t\geq 0\}$ generated by $Y$.
\end{proposition}
\begin{proof}
Fix arbitrary $n \in \mathbb{N}$, $s\geq 0$, $0\leq s_1 \leq ...\leq s_n\leq s$, $\psi_1, ..., \psi_n \in C_{c,S}^\infty$ and  $\Psi \in C_b(\mathbb{R}^n)$, and define the function $\mathcal{I}: D([0,T]; (C_{c,S}^\infty)^*) \to \mathbb{R}$ as
\begin{equation*}
    \caI(X) := \Psi\left(X_{s_1}(\psi_1), ..., X_{s_n}(\psi_n)\right).
\end{equation*}
To show that $\scrM_t^{\phi}(Y)$ and $\scrN_t^{\phi}(Y)$ are $\mathscr{F}_t$-martingales, it suffices to show that
\begin{equation*}
    \lim_{k\to\infty} \E\left[\scrM_t^{N_k,\phi}(Y^{N_k}) \caI(Y^{N_k})\right] = \E\left[\scrM_t^{\phi}(Y) \caI(Y)\right], \ \ \ \ \lim_{k\to\infty} \E\left[\scrN_t^{N_k,\phi}(Y^{N_k}) \caI(Y^{N_k})\right] = \E\left[\scrN_t^{\phi}(Y) \caI(Y)\right],
\end{equation*}
with $\scrM_t^{N,\phi}$ and $\scrN_t^{N,\phi}$ the Dynkin martingales defined in \eqref{dynk mart}. Namely, by the martingale property we  then have that
\begin{equation*}
    \E\left[\scrM_t^{\phi}(Y)\caI(Y)\right] = \lim_{k\to\infty} \E\left[\scrM_t^{N_k,\phi}(Y^{N_k})\caI(Y^{N_k})\right] =  \lim_{k\to\infty} \E\left[\scrM_s^{N_k,\phi}(Y^{N_k})\caI(Y^{N_k})\right]= \E\left[\scrM_s^\phi(Y)\caI(Y)\right],
\end{equation*}
and analogous for $\scrN_t^{\phi}(Y)$.

We start by proving  $\scrM_t^\phi(Y)$ is a martingale. First of all, note that from Proposition \ref{substitute M} we can conclude
\begin{equation*}
    \lim_{k\to\infty} \E\left[\scrM_t^{N_k,\phi}(Y^{N_k}) \caI(Y^{N_k})\right] = \lim_{k\to\infty} \E\left[\scrM_t^{\phi}(Y^{N_k}) \caI(Y^{N_k})\right].
\end{equation*}
Furthermore, in Lemma \ref{substitute M2} we have shown that  the process $\scrM_{t}^\phi(Y^N)$ is uniformly bounded in $L^4$, hence it is also uniformly bounded in $L^2$, therefore
\begin{equation*}
    \sup_{k\in\mathbb{N}} \E\left[\left|\scrM_t^{\phi}(Y^{N_k}) \caI(Y^{N_k})\right|^2\right] \leq ||\Psi||_\infty^2 \sup_{k \in \mathbb{N}} \sum_{\sigma \in S}  \E\left[\left(\scrM_{t}^{\phi}(Y^{N_k})\right)^2\right] <\infty.
\end{equation*}
This implies that we have uniform integrability of $\scrM_t^{\phi}(Y^{N_k}) \caI(Y^{N_k})$. It now suffices to show that $\scrM_t^{\phi}(Y^{N_k}) \caI(Y^{N_k})$ converges to $\scrM_t^{\phi}(Y) \caI(Y)$ in distribution. One usually concludes this using the  Portmanteau theorem, but because
the path space $D([0,T]; (C_{c,S}^\infty)^*_S)$ is not metrizable, we cannot directly use this.  Instead, using the exact same method as introduced in \cite[Proposition 5.2]{Ginkel}, one can work around the problem of non-metrizability via the continuity of $t\mapsto Y_t(\phi)$.

The proof that $\scrN_t^\phi(Y)$ is a martingale works in the same way. First we note that by Proposition \ref{substitute N} we have that
\begin{equation*}
    \lim_{k\to\infty}\E\left[\scrN_t^{N_k,\phi}(Y^{N_k})\caI(Y^{N_k})\right] = \lim_{k\to\infty} \E\left[\scrN_t^{\phi}(Y^{N_k})\caI(Y^{N_k})\right].
\end{equation*}
Therefore we only need to show that
\begin{equation}\label{scrN}
     \sup_{k\in\mathbb{N}} \E\left[\left|\scrN_t^{\phi}(Y^{N_k}) \caI(Y^{N_k})\right|^2\right] < \infty.
\end{equation}
Afterwards the convergence  of $\scrN_t^{\phi}(Y^{N_k}) \caI(Y^{N_k})$  to $\scrN_t^{\phi}(Y) \caI(Y)$ in distribution follows from the same arguments as above.

To see that \eqref{scrN} holds, note that
\begin{align*}
     \E\left[\left(\scrN_{t}^{\phi}(Y^{N_k})\right)^2\right] \leq 2 \E\left[\left(\scrM_{t}^{\phi}(Y^{N_k})\right)^4\right] + 8t^2\rho^2\left(\kappa \left<\left<\partial_x\phi,\partial_x\phi\right>\right> +\left<\left<\phi,\Sigma \phi\right>\right>\right)^2.
\end{align*}
In the proof of Lemma \ref{substitute M2}, we have already shown that $\E\left[\big(\scrM_{t}^{\phi}(Y^{N})\big)^4\right]$ is uniformly bounded in $N$, hence the result follows.
\end{proof}

 \nocite{casini}

\noindent\textbf{Acknowledgement }\ The authors would like to thank Christian Maes for helpful discussions. \\

\appendix
\section{Hydrodynamic limit}\label{appendix}
In this section we give the proof of the hydrodynamic limit, i.e., of Theorem \ref{weak solution}. We follow the standard methodology of \cite{timo}.
\subsection{Preliminary results}\label{5.2}
Before we start the proof of Theorem \ref{weak solution}, we first show the following lemma which, using duality, provides uniform upper bounds for the first and second moment of the expected particle number when starting from the local equilibrium distribution \eqref{loceq}.
\begin{lemma}\label{ineqs}
For all $N\in\mathbb{N}$, $t\geq0$ and $(x,\sigma) \in V$ we have that
\begin{equation}\label{ineq 1}
\E_{\mu_\rho^N} \left[ \eta_t^N(x,\sigma)\right] \leq ||\rho||_\infty,
\end{equation}
\begin{equation}\label{ineq 2}
\E_{\mu_\rho^N} \left[ \eta_t^N(x,\sigma)^2\right] \leq ||\rho||_\infty^2 + ||\rho||_\infty.
\end{equation}
\end{lemma}
\begin{proof}
For the first inequality, note that by duality we have that
\begin{align*}
\E_{\mu_\rho^N} \left[ \eta_t^N(x,\sigma)\right]
&=  \int \E_{\eta^N} \left[D(\delta_{(x,\sigma)}, \eta_t^N) \right] \dd\mu_\rho^N(\eta^N)\\
&=\int \widehat{\E}_{(x,\sigma)} \left[D(\delta_{(\widehat{X}_t,\widehat{\sigma}_t)}, \eta^N) \right] \dd\mu_\rho^N(\eta^N)\\
&=\widehat{\E}_{(x,\sigma)} \left[\rho(\tfrac{\widehat{X}_t}{N},\widehat{\sigma}_t) \right] \leq ||\rho||_\infty.
\end{align*}
Similarly for the second inequality, we have that
\begin{align*}
\E_{\mu_\rho^N} \left[ \eta_t^N(x,\sigma)^2\right]
&= \int \E_{\eta^N} \left[D(2\delta_{(x,\sigma)}, \eta_t^N) + D(\delta_{(x,\sigma)},\eta_t^N)\right]\dd\mu_\rho^N(\eta^N) \\
&= \int \widehat{\E}_{(x,\sigma),(x,\sigma)}\left[D(\delta_{\widehat{X}^{(1)}_t,\widehat{\sigma}^{(1)}_t)} + \delta_{(\widehat{X}^{(2)}_t,\widehat{\sigma}^{(2)}_t)}, \eta^N) + D(\delta_{(\widehat{X}^{(1)}_t,\widehat{\sigma}^{(1)}_t)},\eta^N)\right] \dd\mu_\rho^N(\eta^N)\\
&=\widehat{\E}_{(x,\sigma),(x,\sigma)}\left[\rho(\tfrac{\widehat{X}^{(1)}_t}{N},\widehat{\sigma}^{(1)}_t) \rho(\tfrac{\widehat{X}^{(2)}_t}{N},\widehat{\sigma}^{(2)}_t)  + \rho(\tfrac{\widehat{X}^{(1)}_t}{N},\widehat{\sigma}^{(1)}_t)  \right] \leq||\rho||_\infty^2 + ||\rho||_\infty.
\end{align*}
\end{proof}
Now we will define the processes $ \scrM^{\phi}_t(\pi^N)$ and $ \scrM^{N, \phi}_t(\pi^N)$ the same way as in \eqref{mart prob} and \eqref{dynk mart} respectively. We will again show that we can exchange these processes in the limit.
\begin{proposition}\label{prop dynkin hydro}
For all $t\geq0$ and $\phi \in C_{c,S}^\infty$, we have that
\[
\lim_{N\to\infty} \E\left[\left| \scrM^{\phi}_t(\pi^N) - \scrM^{N,\phi}_t(\pi^N)\right|\right] = 0.
\]
\end{proposition}
\begin{proof}
Through similar calculations as in the proof of Proposition \ref{substitute M}, we find that
\begin{equation}\label{Ln+R}  L_N\pi_{s}^N(\phi) = \pi_{ s}^N( A\phi)+ R_3(\phi, N, s).
\end{equation}
Here $R_3(\phi, N, s)$ is the error term of the Taylor approximations, which is bounded as follows
\begin{equation}\label{rest}
\left|R_3(\phi, N,s)\right| \leq  \frac{1}{N^2} \sum_{(x,\sigma) \in V_N} \eta_{s}^N(x,\sigma)\big(\kappa||\phi_{xxx}||_\infty  + \lambda\sigma^2||\phi_{xx}||_\infty\big),
\end{equation}
and so by \eqref{ineq 1}
\begin{align*}
\lim_{N\to\infty} \E\left[\left| \scrM^{\phi}_t(\pi^N) - \scrM^{N,\phi}_t(\pi^N)\right|\right]
&= \lim_{N\to\infty}  \int_0^t \E\big[\left|R_3(\phi, N,s)\right| \big]\dd s\\
&\leq \lim_{N\to\infty}  \frac{1}{N^2} t|V_N|\cdot  ||\rho||_\infty\big(\kappa ||\phi_{xxx}||_\infty  + \lambda \sigma^2||\phi_{xx}||_\infty\big)\\
&=0,
\end{align*}
which concludes the proof.
\end{proof}
Lastly we will prove that the martingale $ \scrM^{N,\phi}_t(\pi^N)$ actually vanishes in the limit.
\begin{lemma}\label{vanish martingale}
For any $\phi \in C_{c,S}^\infty$ we have that
$$\lim_{N\to\infty} \E\left[\sup_{t\in[0,T]} \left|\scrM_{t}^{N,\phi}(\pi^N)\right|^2\right] = 0.$$
\end{lemma}
\begin{proof}
First of all, by Doob's maximal inequality, we have that
$$ \E\left[\sup_{t\in[0,T]} \left|\scrM_{t}^{N,\phi}(\pi^N)\right|^2\right] \leq 4\E\left[\left(\scrM_{T}^{N,\phi}(\pi^N)\right)^2\right].$$
 Since $\scrM_{t}^{N,\phi}(\pi^N)$ is a mean-zero martingale, this expectation is equal to the expectation of the quadratic variation of  $ \scrM_{t}^{N,\phi}(\pi^N)$, i.e.,
\[
\E\left[\left( \scrM_{T}^{N,\phi}(\pi^N)\right)^2\right] = \E\left[\int_0^T \Gamma_{s}^{N,\phi}(\pi^N)\dd s\right]= \int_0^T \E\left[\Gamma_{s}^{N,\phi}(\pi^N)\right]\dd s,
\]
where $\Gamma_{s}^{N,\phi}$ is as defined in \eqref{Gamma}. By using the same calculations to get \eqref{gamma 1 deriv} we find that
 \begin{align*}
 \Gamma_{s}^{N,\phi}(\pi^N)
         &= \frac{2\kappa}{N^2} \sum_{(x,\sigma) \in V}\eta^N_s(x,\sigma) (\partial_x\phi(\tfrac{x}{N},\sigma))^2  +\frac{1}{N^2} \sum_{(x,\sigma)\in V}\sum_{\sigma'\in S} c(\sigma,\sigma')\eta^N_s(x,\sigma)(\phi(\tfrac{x}{N},\sigma')-\phi(\tfrac{x}{N},\sigma))^2\\
         &\ \ \ \ \ \ \ \  + R_4(\phi,s,N,\sigma),
\end{align*}
with $R_4(\phi,s,N)$ bounded as follows
\ben
|R_4(\phi,s,N)| \leq \kappa\frac{1}{N^4}\sum_{(x,\sigma)\in V_N} \eta^N_s(x,\sigma)(\kappa||\phi_{xx}||_\infty + \lambda \sigma ||\phi_{x}||_\infty).
\een
By dominated convergence and \eqref{ineq 1} we can then conclude that
$$\lim_{N\to\infty} \E\left[ \left( \scrM_{T}^{N,\phi}(\pi^N)\right)^2\right] = \lim_{N\to\infty} \int_0^T \E \left[\Gamma_{s}^{N,\phi}(\pi^N)\right]\dd s = 0,$$
and the result follows.
\end{proof}

\subsection{Tightness}\label{Tightness}
We now prove the tightness result for the sequence $\{\pi^N: N \in \mathbb{N}\}$.
\begin{proposition}\label{tight probs}
$\{\pi^N: N \in \mathbb{N}\}$ is tight in $D( [0,T];\textbf{\emph{M}})$.
\end{proposition}
\begin{proof}
In the space $D( [0,T];\textbf{M})$ we can prove tightness by showing that the following two assertions hold.
\begin{enumerate}[label=\textbf{B.\arabic*}]
    \item \label{D1}For all $t\in[0,T]$ and $\varepsilon>0$ there exists a compact $K(t,\varepsilon) \subset \textbf{M}$ such that
    \begin{equation*}
        \sup_{N\in \mathbb{N}} \mathbb{P}\big(\pi_t^N \notin K(t,\varepsilon)\big) \leq \varepsilon.
    \end{equation*}
    \item \label{D2} For all $\varepsilon>0$
    \begin{equation*}
       \lim_{\delta \to 0} \limsup_{N\to\infty} \mathbb{P}\big(\omega(\pi^N,\delta)\geq \varepsilon\big) = 0,
    \end{equation*}
   where
$$\omega(\alpha, \delta) =\sup\{d \big(\alpha_s,\alpha_t\big) : s,t \in [0,T], |t-s|<\delta\},$$
and $d$ is the metric on  $\textbf{M}$ given by
\[
 d \big(\alpha,\beta \big)=\sum_{j=1}^\infty 2^{-j}\left(1 \wedge \left|\alpha (\phi_j) - \beta (\phi_j)\right|\right)
\]
for some specific choice of test functions $\phi_j \in C_{c,S}^\infty$.
\end{enumerate}
We start by proving \ref{D1}. Fix $\varepsilon>0$ and $t\geq 0$, and for some $C >0$ let $K_C$ be the following set
$$K_C = \left\{\mu \in \textbf{M}: \mu([-k,k]\times S) \leq C(2k+1)k^2 \text{ for all $k\in \mathbb{N}$}\right\}.$$
By\cite[Proposition A.25]{timo}, this is a compact set in $\textbf{M}$, and by Markov's inequality we now have that
\begin{align*}
\p(\pi_{t}^N([-k,k]\times S)\geq C(2k+1)k^2) &\leq \frac{1}{C(2k+1)k^2} \E\left[\pi_{t}^N([-k,k]\times S)\right]\\
&= \frac{1}{C(2k+1)k^2N} \sum_{(x,\sigma) \in [-kN,kN]\times S} \E\left[ \eta^N_{t}(x,\sigma)\right]\\
&\leq  \frac{1}{C(2k+1)k^2N} (2k+1)N|S|\cdot ||\rho||_\infty\\
&= \frac{1}{Ck^2}|S| \cdot  ||\rho||_\infty.
\end{align*}
Here we have used the inequality in \eqref{ineq 1}. Therefore
\begin{align*}
\p(\pi_{t}^N \notin K_C) \leq \sum_{k =1}^\infty \p(\pi_{t}^N([-k,k])\geq C(2k+1)k) \leq \frac{1}{C} |S|\cdot ||\rho||_\infty \sum_{k=1}^\infty \frac{1}{k^2}<\infty.
\end{align*}
By now taking $C$ big enough, we then have that for all $N\in\mathbb{N}$ that $\p(\pi_{t}^N \notin K_C) \leq\varepsilon$, which finishes the proof of \ref{D1}.

In order to prove that \ref{D2} holds, note first that
\begin{equation}\begin{split}\label{omega delta}
\omega(\pi^N, \delta)
& = \sup_{\substack{s,t \in [0,T]\\|t-s|<\delta}} \sum_{j=1}^\infty 2^{-j}\left(1 \wedge \left|\pi^N_{ t}(\phi_j) - \pi^N_{ s}(\phi_j)\right|\right)\\
& \leq 2^{-m} + \sum_{j=1}^m  \sup_{\substack{s,t \in [0,T]\\|t-s|<\delta}} 2^{-j}\left(1 \wedge \left|\pi^N_{ t}(\phi_j) - \pi^N_{ s}(\phi_j)\right|\right)\\
& \leq  2^{-m} + \sum_{j=1}^m  \sup_{\substack{s,t \in [0,T]\\|t-s|<\delta}}  \left|\pi^N_{ t}(\phi_j) - \pi^N_{ s}(\phi_j)\right|.
\end{split}\end{equation}
Here we have taken $m$ arbitrarily, so the first term can be made as small as we want. We now want to show that the expecation of the sum vanishes as we let $N\to \infty$ and $\delta \downarrow 0$. Afterwards, the claim can be shown by using the Markov inequality.

Note first that we have the following,
\begin{equation}\begin{split}\label{sup}
\E\left[  \sup_{\substack{s,t \in [0,T]\\|t-s|<\delta}}  \left|\pi^N_{t}(\phi_j) - \pi^N_{s}(\phi_j)\right|^2\right]
& = \E\left[ \sup_{\substack{s,t \in [0,T]\\|t-s|<\delta}} \left|\scrM_{t}^{N,\phi_j}(\pi^N) - \scrM_{s}^{N,\phi_j}(\pi^N) - \int_s^t L_N \pi_{r}^N(\phi_j)\dd r \right|^2\right] \\
&  \leq 4\E\left[ \sup_{t \in [0,T]} \left(\scrM_{t}^{N,\phi_j}(\pi^N)\right)^2\right] + 2\E\left[\sup_{\substack{s,t \in [0,T]\\|t-s|<\delta}} \left| \int_s^t L_N \pi_{r}^N(\phi_j)\dd r\right|^2\right].\\
\end{split}\end{equation}
 By Lemma \ref{vanish martingale}, the first term goes to zero as $N\to \infty$. For the second term, by filling in \eqref{Ln+R} we find that
\begin{equation}\begin{split}\label{abc}
\left| \int_s^t L_N \pi_{r}^N(\phi_j)\dd r\right|^2 &= \left( \int_s^t \left(\pi_{r}^N(A\phi_j) + R_3(\phi_j,N,r)\right) \dd r\right)^2\\
&\leq 2\left( \int_s^t \pi_{r}^N(A\phi_j) \dd r\right)^2 + 2\left(\int_s^t R_3(\phi_j,N,r,\sigma)\dd r\right)^2.
\end{split}\end{equation}
 By the upper bound on $R_3(\phi_j,N,r)$ in \eqref{rest} and by \eqref{ineq 1}, we can see that the last term vanishes in expectation when $N\to\infty$.
For the other term we have that
\begin{align*}
\left( \int_s^t \pi_{r}^N(A\phi_j) \dd r\right)^2 =\frac{1}{N^2} \left[ \int_s^t \sum_{(x,\sigma) \in V} \eta_r^N(x,\sigma)\cdot  (A\phi_j)(\tfrac{x}{N},\sigma)\dd r\right]^2.
\end{align*}
 Using that $|t-s|<\delta$ and applying H\"older a number of times, we find that
\begin{align*}
\left( \int_s^t \pi_{r}^N(A\phi_j) \dd r\right)^2 &\leq \frac{1}{N^2}  |V_{\phi_j}^N| \delta \cdot ||A\phi||_\infty \sum_{(x,\sigma) \in V_{\phi_j}^N} \int_0^T \left(\eta_{r}^N(x,\sigma) \right)^2\dd r.
\end{align*}
Using the inequality in \eqref{ineq 2}, we find that
\begin{align*}
\E\left[ \sup_{\substack{s,t \in [0,T]\\|t-s|<\delta}} \left( \int_s^t \pi_{r}^N(A\phi_j) \dd r\right)^2 \right]   \leq  \frac{1}{N^2}  |V_{\phi_j}^N|^2 \delta T \cdot ||A\phi||_\infty (||\rho||_\infty^2 + ||\rho||_\infty) = \mathcal{O}(\delta).
\end{align*}
Therefore
\begin{equation}\label{limsup}
\lim_{\delta \downarrow 0} \limsup_{N\to\infty} \E\left[  \sup_{\substack{s,t \in [0,T]\\|t-s|<\delta}}  \left|\pi^N_{t}(\phi_j) - \pi^N_{s}(\phi_j)\right|^2\right]= \lim_{\delta \downarrow 0} \limsup_{N\to\infty}\E\left[ \sup_{\substack{s,t \in [0,T]\\|t-s|<\delta}} \left( \int_s^t  \pi_{r}^N(A\phi_j) \dd r\right)^2 \right]=0.
\end{equation}
So, by going back to \eqref{omega delta} and using the Markov inequality, we get the following:
\begin{align*}
 \p( \omega(\pi^N, \delta)\geq \varepsilon) \leq \frac{1}{\varepsilon} \left(2^{-m}  + \sum_{j=1}^m \E \left[\sup_{\substack{s,t \in [0,T]\\|t-s|<\delta}}  \left|\pi^N_{ t}(\phi_j) - \pi^N_{ s}(\phi_j)\right|\right]\right).
\end{align*}
By now taking $m$ such that $2^{-m} < \varepsilon^2$ and using \eqref{limsup} we see that
\begin{align*}
 &\lim_{\delta\downarrow 0} \limsup_{N\to\infty} \p( \omega(\pi^N, \delta)\geq \varepsilon)< \varepsilon,
\end{align*}
which ultimately proves the tightness result.
\end{proof}

\subsection{Proof of hydrodynamic limit}\label{5.5}
Now we have everything needed to prove the result.
\begin{proof}[Proof of Theorem \ref*{weak solution}]\ \\
From  the tightness of the sequence $\{P^N : N \in \mathbb{N}\}$ we know that there exists a subsequence $\{P^{N_k}: k \in \mathbb{N}\}$ that converges weakly in the Skorokhod topology, i.e., $P^{N_k} \xrightarrow{w} P$ for some probability measure $P$ on $D([0,T];\textbf{M})$. If we can show that every convergent subsequence converges to the dirac measure $P=\delta_{\pi}$ with $\pi$ the unique continuous path solving \eqref{pi}, then the result follows.

First of all, by \ref{D2}, we immediately know that $P$ is concentrated on continuous paths in $D([0,T];\textbf{M})$. Now define for $\phi\in C_{c,S}^\infty$, $\varepsilon>0$  and $T>0$ the following set
$$H(\phi,\varepsilon):= \left\{\alpha \in D([0,T];\textbf{M}) : \sup_{t\in[0,T]} \left| \alpha_t(\phi) - \alpha_0(\phi) - \int_0^t\alpha_s(A\phi)\dd s\right|\leq\varepsilon\right\}.$$
Analogously as in \cite[Lemma 8.7]{timo} one can prove that this set is closed in the Skorokhod topology. Since the set $H(\phi,\varepsilon)$ is closed, we can apply the Portmanteau Theorem  to see that
\begin{align*}
P\big(H(\phi,\varepsilon)\big) &\geq \limsup_{k\to\infty} P^{N_k}\big(H(\phi,\varepsilon)\big)\\
&= \limsup_{k\to\infty} \p\left(\sup_{t\in[0,T]}\left| \pi^{N_k}_{t}(\phi) - \pi^{N_k}_{0}(\phi) - \int_0^t \pi^{N_k}_{s}(A\phi)\dd s \right|\leq \varepsilon \right)\\
&=\limsup_{k\to\infty} \p\left( \sup_{t\in[0,T]}\left| \scrM_{t}^{\phi}(\pi^{N_k})\right| \leq \varepsilon\right)\\
&=\limsup_{k\to\infty} \p\left( \sup_{t\in[0,T]}\left| \scrM_{t}^{N_k,\phi}(\pi^{N_k})\right|\leq \varepsilon\right).
\end{align*}
Here we have used Proposition \ref{prop dynkin hydro} for the last equality. By Lemma \ref{vanish martingale} and the Markov inequality we then have that
$$\p\left( \sup_{t\in[0,T]}\left| \scrM_{t}^{N_k,\phi}(\pi^{N_k})\right| >\varepsilon\right)\leq \frac{1}{\varepsilon^2} \E\left[\sup_{t\in[0,T]}\left|\scrM_{t}^{N_k,\phi}(\pi^{N_k})\right|^2\right] \to 0,$$
so $P\big(H(\phi,\varepsilon)\big)  = 1$. Since we took $\varepsilon>0$ arbitrarily, we indeed find that $P=\delta_{\pi}$.
\end{proof}
\end{document}